\documentclass[12pt,a4paper, oneside, regno]{amsart} %12pt scrbook
\usepackage[utf8]{inputenc}
\usepackage[english]{babel}
\usepackage[T1]{fontenc}

\usepackage{braket}
\usepackage{stmaryrd}
\usepackage{amsmath}
\usepackage{amsthm}
\usepackage{comment}
\usepackage{amsfonts}
\usepackage{amssymb}
\usepackage{multicol}
\usepackage{xcolor}
\numberwithin{equation}{section}

\usepackage{url}
\usepackage{pdfpages}

\usepackage{todonotes}

\usepackage{lscape}% landscape orientation
\usepackage{graphicx}
\usepackage{tikz-qtree}
\usepackage[totoc]{idxlayout}
\usepackage{enumitem}   
\usepackage[intoc,refpage]{nomencl}

\usepackage{soul} %Pour souligner ou barrer en couleur

\usepackage{graphicx}
\graphicspath{ {images/} }
\usepackage[all,cmtip]{xy}

\usepackage[toc,page]{appendix}
\usepackage{pgf,tikz}
\usetikzlibrary{decorations.pathreplacing,calligraphy} %for brace
 
\usepackage[font=small,skip=0pt]{caption}

\usetikzlibrary{arrows}
\makeatletter
\newcommand\ackname{Acknowledgements}
\if@titlepage
   \newenvironment{acknowledgements}{%
       \titlepage
       \null\vfil
       \@beginparpenalty\@lowpenalty
       \begin{center}%
         \bfseries \ackname
         \@endparpenalty\@M
       \end{center}}%
      {\par\vfil\null\endtitlepage}
\else
   
\fi
\makeatother
\def\restriction#1#2{\mathchoice
              {\setbox1\hbox{${\displaystyle #1}_{\scriptstyle #2}$}
              \restrictionaux{#1}{#2}}
              {\setbox1\hbox{${\textstyle #1}_{\scriptstyle #2}$}
              \restrictionaux{#1}{#2}}
              {\setbox1\hbox{${\scriptstyle #1}_{\scriptscriptstyle #2}$}
              \restrictionaux{#1}{#2}}
              {\setbox1\hbox{${\scriptscriptstyle #1}_{\scriptscriptstyle #2}$}
              \restrictionaux{#1}{#2}}}
\def\restrictionaux#1#2{{#1\,\smash{\vrule height .8\ht1 depth .85\dp1}}_{\,#2}}

\usepackage{mathtools}

\DeclareMathOperator{\tp}{tp}

\DeclareMathOperator{\val}{val}

\DeclareMathOperator{\Th}{Th}

\DeclareMathOperator{\supp}{supp}
\DeclareMathOperator{\Inv}{Inv}

\renewcommand{\div}{\mathop{div}}

\usepackage{scalerel}[2016/12/29]
\makeatletter
\DeclareRobustCommand\bigop[1]{%
  \mathop{\vphantom{\sum}\mathpalette\bigop@{#1}}\slimits@
}
\newcommand{\bigop@}[2]{%
  \vcenter{%
    \sbox\z@{$#1\sum$}%
    \hbox{\resizebox{\ifx#1\displaystyle.9\fi\dimexpr\ht\z@+\dp\z@}{!}{$\m@th#2$}}%
  }%
}
\makeatother

\newcommand{\hprod}{\DOTSB\bigop{\mathrm{H}}}

\theoremstyle{plain}
\newtheorem{theorem}{Theorem}[section]
\newtheorem{lemma}[theorem]{Lemma}
\newtheorem{proposition}[theorem]{Proposition}
\newtheorem{claim}{Claim}
\newtheorem{remark}[theorem]{Remark}

\newtheorem{observation}[theorem]{Observation}
\newtheorem{fact}[theorem]{Fact}
\newtheorem{maintheorem}[theorem]{Main Theorem}
\theoremstyle{definition}
\newtheorem{definition}[theorem]{Definition}
\newtheorem{corollary}[theorem]{Corollary}

\newtheorem{question}[theorem]{Question}
\newtheorem*{Notation}{Notation}
\theoremstyle{remark}
\newtheorem*{remark*}{Remark}

\definecolor{black}{rgb}{0,0,0}
\newtheorem{example}[theorem]{Example}

\newcommand{\Z}{\mathbb{Z}}
\newcommand{\Q}{\mathbb{Q}}

\title{Stably Embedded Pairs of Ordered Abelian Groups}
\author{Martin Hils}
\address{Institut f\"{u}r Mathematische Logik und Grundlagenforschung, Universit\"{a}t M\"{u}nster, Einsteinstr. 62, D-48149 M\"{u}nster, Germany}
\email{hils@uni-muenster.de}
\thanks{MH was partially supported by the German Research Foundation (DFG) via HI 2004/1-1 (part of the French-German ANR-DFG project GeoMod) and under Germany's Excellence Strategy EXC 2044-390685587, `Mathematics M\"unster: Dynamics-Geometry-Structure'.}

\author{Martina Liccardo}
\address{Dipartimento di Matematica e Applicazioni "Renato Caccioppoli", Università degli Studi di Napoli Federico II}

\author{Pierre Touchard}
\address{Dipartimento di matematica e fisica, Università degli Studi della campania ``Luigi Vanvitelli" .}
\curraddr{Institut für Algebra, Technische Universität Dresden, 01062 Dresden, Germany}
\email{pierre.pa.touchard@gmail.com}

\thanks{PT was supported by a grant of the University of Campania ?Luigi Vanvitelli? in the framework of V:ALERE 2019 (GoAL project) and partially supported by KU Leuven IF
C14/17/083 and C16/23/010.
}

\date{\today}

\begin{document}

\begin{abstract}
        We investigate when an ordered abelian group $G$ is stably embedded in a given elementary extension $H$. We focus on a large class of ordered groups which includes maximal ordered groups with interpretable archimedean valuation. We give a complete answer for groups in this class which takes the form of a transfer principle for valued groups. It follows in particular that all types over the lexicographic product $\prod_{i\in \omega} \mathbb{Z}$ are definable. 
\end{abstract}

\maketitle

%\tableofcontents

\section*{Introduction}
Every stable structure $\mathcal{M}$ satisfies the following property: all externally definable sets -- sets definable with parameters in an elementary extension -- are definable internally -- i.e., with parameters from $\mathcal{M}$. This property of $\mathcal{M}$, called here \textit{stable embeddedness}, is not equivalent to stability. For instance, the field of real numbers $(\mathbb{R},+,\cdot,0,1)$ and Presburger arithmetic $(\mathbb{Z},+,0,<)$ are both stably embedded structures (see, e.g., \cite{VdD86} and  \cite{CV}, respectively). More generally, it is of interest to characterise in the context of an unstable theory the stably embedded elementary pairs $(\mathcal{N},\mathcal{M})$, that is the elementary extensions $\mathcal{M}\preccurlyeq \mathcal{N}$ such that any trace of an $\mathcal{N}$-definable set in $\mathcal{M}$ is $\mathcal{M}$-definable. %By the duality of parameters and variables, it is equivalent to ask to characterise structure $S$ where all types over $\mathcal{N}$ are definable in $\mathcal{M}$. 

\subsection*{Space of definable types}
The space of definable types is known to carry useful information of the structure. In various contexts, it has been shown that this space can be equipped with a (strict) pro-definable structure, which allows for a finer model theoretic analysis  with various tools. The pro-definability of the space of definable types has been recently quite intensively studied in \cite{CD16,CY21,CHY26}. In particular, the first author together with Cubides and Ye extended in \cite{CHY26} Poizat's notion of beautiful pairs to an unstable context, establishing a close relationship between the axiomatisability of these pairs with the strict prodefinability of the space of definable types. In this framework, it is primordial to discuss the uniform definability of definable types and the axiomatisation of stably embedded pairs.  In this paper, we contribute to this framework by studying stably embedded pairs of ordered abelian groups. 
This contribution also extends a previous work of the third author in \cite{Tou20b}, where the analysis of elementary pairs of stably embedded (henselian) benign valued fields is reduced to that of pure fields and pure ordered abelian groups.

\subsection*{Classes of ordered abelian groups} 
%We also characterise when an elementary pairs of ordered abelian groups elementary equivalent to a pseudo-complete ordered abelian group is stably embedded. 

The class of ordered abelian groups (OAG) has been intensively studied, and despite its apparent simplicity, lengthy and difficult analyses are often required in order to study its model theoretic properties. In the literature, authors tend to restrict their study to the much smaller classes of ordered abelian groups of \emph{finite regular rank}  (FRR, with finitely many definable convex subgroups; see Subsection~\ref{S:finite-rk}). 
A larger well studied subclass consists of ordered abelian groups with \emph{finite spines} (FS, i.e. of finite $n$-regular rank for all $n$). For instance, elimination of imaginaries is known in FS (see \cite{Vic22}), as well as a characterisation of the distality (see \cite{ACGZ20}) but these results have not yet been extended beyond this class. Halevi and Hasson compute dp-rank in all ordered abelian groups in \cite{HH19A}. Note that this computation is only interesting within the class FS, as no ordered abelian group with an infinite spine is strong NIP. 

In this paper, we will focus on another large class of ordered abelian groups, where a certain interpretable valuation, called the regular valuation, plays an important role. This class consists of ordered abelian groups satisfying a certain property denoted by (UR) (see Subsection~\ref{S:UniformValuation}). In some sense, this class is orthogonal to the class FS -- one can show that their intersection is exactly FRR-- and it contains maximal groups with interpretable archimedean valuation such as $\sum_\omega \mathbb{Z}$. 
In addition to this property (UR), we will assume a certain simplifying property, denoted by (M), which is a first order pendant of maximality for the regular valuation. %In particular, an ordered abelian group satisfying (UR) and (M) is elementary equivalent to a maximal group. 
Our approach to study this class is to apply methods from valuation theory.

\begin{figure}
    \centering
\begin{tikzpicture}[scale=1.5]
    \begin{scope}
  \clip (0,-0.48) ellipse (1.5 and 2);
  \fill[red!20,rotate=-45] (1.24,0) ellipse (1 and 2);
\end{scope}
    \draw (0,2.2) node{OAG};

    \draw (-1.9,-0.3) node{FS};
    \draw (0,-1.5) node{FRR};
        \draw (0,0.8) node{(M)};
    \draw (1.9,-0.3) node{(UR)};
    %\draw (1,-0.9) node{(M)};
    \draw[rotate=45] (-1.24,0) ellipse (1 and 2);
    \draw[rotate=-45](1.24,0) ellipse (1 and 2);
    \draw  (0,-0.48) ellipse (1.5 and 2);
    \draw (-2.47,0.05) arc (180:0:2.47);

\end{tikzpicture}

    \caption*{In red, the class of ordered abelian groups studied in this paper.}
\end{figure}

%Our study of stably embedded elementary pairs of ordered abelian groups includes the class FRR. However, we will not study stably embedded pairs in FS.  This class contains important and natural examples of ordered abelian groups, but remains quite restrictive. 

\subsection*{Relative quantifier elimination}
To study stable embeddedness of elementary pairs of such ordered abelian groups, our main tool is a result of relative quantifier elimination. A language for quantifier elimination in the class of all ordered abelian groups is known since the work of Gurevich and Schmitt \cite{GS84}. This work has been revisited and simplified by Cluckers and Halupczok \cite{CH11}, who gave a multisorted language $\mathcal{L}_{syn}$ for relative quantifier elimination. It involves a main sort for the group itself, and sorts for the \emph{(definable) spines} - these are chains of certain uniformly definable families of convex subgroups. Despite its apparent complexity, the language $\mathcal{L}_{syn}$ seems to be optimised for the full class of ordered abelian groups and admits good syntactical properties. One can observe that in this language, the set of sorts for the spines is closed in the sense of \cite[Appendix A]{Rid17}. 
For the class of ordered abelian groups which satisfy (UR) and (M), one can prove  a simpler result of relative quantifier elimination, in a language that we simply denote by $\mathcal{L}$ (see Subsection~\ref{SS:QuantifierEliminationUniformValuation}). The property (UR) ensures indeed that one valuation (namely the  regular valuation) and thus one spine (the \emph{regular spine}) suffices to describe the convex subgroups, and the property (M) allows to avoid some pathological behavior with respect to this valuation. A natural consequence is that the induced structure on the regular spine is that of a pure coloured chain $(\Gamma,<,(C_\phi)_\phi)$, i.e. a total order endowed with (countably many) unary predicates $C_\phi$.

This language $\mathcal{L}$ also  allows us to highlight the fundamental role that valuations play in the study of ordered abelian groups. 
In particular, it exhibits the role of the  \emph{regular ribs} (analogous to that of the residue field in valued fields). 
Indeed, we use this result of relative quantifier elimination to prove our main theorem, which takes the form of a transfer principle for valued groups down to the spine and the ribs.

\subsection*{Transfer principles in ordered abelian groups}
It is well understood that relative quantifier elimination in the context of valued fields leads to (other) transfer principles: for instance from the theorem of Pas one deduces that henselian valued fields of equicharacteristic $0$ are complete and model-complete (resp. NIP) relative to the value group and residue field -- this is the famous Ax-Kochen-Ershov (resp. Delon's) principle.

Analogously, one can see that relative quantifier elimination in valued abelian groups leads to transfer principles for the group down to the spine (a coloured chain) and the set of ribs (regular ordered abelian groups). %That is, a statement of the form `a certain model theoretic property $P$ holds in a order abelian group if and only if $P$ holds in the ribs and in the spine'. 
Indeed, some well known results hide such a latent transfer principle:
\begin{itemize}
\item Gurevich's decidability result follows from the decidability of the class of regular ordered abelian groups and the decidablity of the class of coloured chains (see \cite{Gur64}).
\item That all ordered abelian groups are NIP (Gurevich-Schmitt) follows from the fact that all regular ordered abelian groups and all coloured chains are NIP.
\end{itemize}
The reader can refer to the introduction of Schmitt in \cite{Sch84}. It is noticeable that, in this pioneering work, the role of the ribs is anecdotal and not explicitly mentioned: if some of their statements mentioned the necessary role of the spines, that of the ribs is often relegated to the proof. There exists however a more explicit transfer principle down to the ribs:

\begin{itemize}
\item In \cite{ACGZ20}, Aschenbrenner, Chernikov, Gehret and Ziegler study  distality in ordered abelian groups with finite spines, by reducing this property to regular ordered abelian groups.
\end{itemize}

Contrarily to the works mentioned above, in this paper, for the first time, we prove a transfer principle for ordered abelian groups which involves simultaneously the ribs and the spines.

\subsection*{Main result and plan of the paper}
We characterise as follows stably embedded ordered abelian groups satisfying (UR) and (M):

{
    \renewcommand{\thetheorem}{(Theorem~\ref{TheoremCharacterisationStableEmbeddednessCaseInterpretableArchimedeanSpineAbsolut})}
\begin{maintheorem}
     An ordered abelian group $(G,+,0,<)$ satisfying (UR) and (M) is stably embedded if and only if it is maximal, its regular ribs ${(R_\gamma,+,0,<)}$ are stably embedded as ordered abelian groups and its regular spine $(\Gamma,(C_\phi)_{\phi\in \mathcal{L}_{\text{oag}}},<)$ is stably embedded as a coloured chain.
\end{maintheorem}

%    \addtocounter{theorem}{-1}
}

Our analysis allows us in fact to study more generally stably embedded elementary pairs of ordered abelian groups (Theorem~\ref{TheoremCharacterisationStableEmbeddednessCaseInterpretableArchimedeanSpineForPairs}). These theorems give new instances of transfer principles for ordered abelian groups, which involve both the regular ribs and the regular spine. It follows for example that the Hahn product $\hprod_{i\in \omega} \mathbb{Z}$ is the unique model of its theory which is stably embedded  (Example~\ref{ExampleHprodZStablyEmbedded}). We deduce as a corollary that an ordered abelian group satisfying (UR) and (M) is stably embedded if and only if all cuts are definable (Corollary \ref{CorMainTheorem2}).

\subsection*{A Kaplansky theory for valued abelian groups}
This approach of using valuations to study ordered abelian groups has been considered on many occasions
(see, e.g., \cite{Hol63}). To some extent, many methods and traditional tools for valued fields can be  adapted for valued groups (see, e.g., \cite{SS91,KuhFV}). However, the literature remains  sparse and incomplete. We take the opportunity here to develop a  Kaplansky theory for valued abelian groups and study maximal and pseudo-complete valued abelian groups. In particular, we show that a $\mathbb{Z}$-invariant valued abelian group is maximal if and only it is pseudo-complete (see Theorem~\ref{TheoremMaximalityEquivalentPseudo-completeness}). 

Some of these results, in particular the one mentioned above concerning $\hprod_{i\in \omega} \mathbb{Z}$, were shown during the PhD of the second author, and partial results, including a variant of the main theorem can be found in her thesis \cite{Lic22}. 

\subsection*{A brief overview of the paper}
In Section~\ref{S:Prelim}, we recall some necessary background: the definition of stable embeddedness and some algebraic and model theoretic facts on ordered abelian groups. We also briefly discuss stably embedded coloured chains. In Subsection~\ref{SS:KaplanskyTheory}, we develop a Kaplansky theory for valued abelian groups and pairs of valued abelian groups. In Subsection~\ref{SS:ValueGroupModm}, we show that completeness of a $\mathbb{Z}$-invariant valued abelian group $(G,\val)$ passes to the quotient valued abelian group $(G/mG,\val^m)$. 

Section~\ref{S:UniformValuation} is devoted to the proof of our main theorem. Subsection~\ref{S:finite-rk} deals with regular ordered abelian groups and ordered abelian groups with finite regular rank. In that context, we extend our analysis to include the study of uniformly stably embedded models. We show in particular that  $\mathbb{Z}^n$ and $\mathbb{Z}^n\times\mathbb{R}$, for $n\in \mathbb{N}$, are the only FRR groups which are \emph{uniformly} stably embedded (see Corollary~\ref{CorollaryZnUniformlySE}).
Then, we define the regular valuation and related objects in Subsection~\ref{SS:RegularValuation} and prove the main theorem in Subsection~\ref{SS:MainTheorem}. 
Finally, in Section~\ref{S:Examples and Counter-examples}, we apply the main theorem to some examples and we give some counter-examples, thus showing that certain hypotheses cannot be dropped.

\section{Notation and preliminaries}\label{S:Prelim}
We begin by presenting some definitions and notions that we will use in this paper. We use the convention that $0 \in \mathbb{N}$ and we write $\mathbb{P}$ for the set of primes. If $\mathcal{M},\mathcal{N} \dots$ are structures, we denote by $M,N,\dots$ their respective base sets.
\subsection{Stable embeddedness}
Let $\mathcal{M}$ be an $\mathcal{L}$-structure, $\mathcal{L}$ any first order language. 
\begin{definition}
Let $\mathcal{N}$ be an elementary extension of $\mathcal{M}$. We say that $\mathcal{M}$ is \emph{stably embedded} in $\mathcal{N}$, and write $\mathcal{M}\preccurlyeq^{st} \mathcal{N}$, if for every definable set $\phi(N^m,\bar{a})$, $\bar{a} \subset N$, its trace $\phi(N^m,\bar{a}) \cap M^m$ is $\mathcal{L}(M)$-definable, i.e., there exist an $\mathcal{L}$-formula $\psi(\bar{x},\bar{z})$ and a tuple $\bar{b}$ of parameters in $M$ such that
\begin{equation}
\phi(N^m,\bar{a}) \cap M^m = \psi(M^m,\bar{b}).
\end{equation}
\end{definition}
Note that $\psi(\bar{x},\bar{z})$ may depend on the parameters $\bar{a}$. If $\psi(\bar{x},\bar{z})$ may be chosen depending only on the formula $\phi(\bar{x},\bar{y})$ and not on $\bar{a}$, $\mathcal{M}$ is said to be \emph{uniformly stably embedded} in $\mathcal{N}$, and we write $\mathcal{M}\preccurlyeq^{ust}\mathcal{N}$.

\begin{definition}
We say that $\mathcal{M}$ is \emph{stably embedded} if $\mathcal{M}$ is stably embedded in every elementary extension. Similarly, we say that $\mathcal{M}$ is \emph{uniformly stably embedded}.
\end{definition} 
These conditions can be equivalently formulated in terms of definability of types. We recall the definition:  
\begin{definition}
A type $p(\bar{x}) \in S_n(M)$ is said to be \emph{definable} if for every $\mathcal{L}$-formula $\phi(\bar{x},\bar{y})$, there exists an $\mathcal{L}(M)$-formula $d_p\phi(\bar{y})$ such that for all $\bar{a} \subset M$
\begin{equation}
p(\bar{x}) \vdash \phi(\bar{x},\bar{a}) \text{ if and only if } M \models d_p\phi(\bar{a}).
\end{equation}
The collection $(d_p\phi)_\phi$ is called a \emph{defining scheme} for $p$.
\end{definition}

\begin{example}
By a fundamental result of Shelah, $T$ is stable if and only if all types over all models of $T$ are definable. Moreover, let $\phi(\bar{x},\bar{y})$ be an $\mathcal{L}$-formula, then there is a formula $\psi(\bar{y},\bar{z})$ such that for every type $p(\bar{x})$ there is $\bar{b}$ such that 
\begin{equation}
d_p\phi(\bar{y})=\psi(\bar{y},\bar{b}).
\end{equation}
In this case, we say that the set of types is \emph{uniformly definable}.
\end{example}

By the duality between object variables and parameter variables, we have:
\begin{fact}
$\mathcal{M}$ is stably embedded in $\mathcal{N}$ if and only if all $n$-types over $M$ realised in $N$ are definable, i.e., for every  $\bar{\alpha} \subset N$, $p(\bar{x})=tp(\bar{\alpha}/M)$ is definable. 

Similarly, $\mathcal{M}$ is uniformly stably embedded in $\mathcal{N}$ if and only if all $n$-types over $M$ realised in $N$ are uniformly definable.
\end{fact}

Let us recall another important result on the definability of types. 

\begin{theorem}[Marker-Steinhorn \cite{MS94}]\label{TheoremMarkerSteinhorn}
Let $T$ be an o-minimal theory, and let $\mathcal{M} \preccurlyeq \mathcal{N}\models T$. Then, all types over $M$ realised in $N$ are uniformly definable if and only if all $1$-types over $M$ realised in $N$ are definable. 
\end{theorem}

We will extend this principle to some class of ordered abelian groups, showing that an elementary pair $\mathcal{M}\preccurlyeq\mathcal{N}$ (in that class) is stably embedded if and only if every 1-type over $M$ realised in $N$ is definable (see Corollary \ref{Cor:StEmbAllCutsDefinable}).

\subsection{Generalities on  chains}\label{SS:GeneralitiesonChains}
\subsubsection*{Cuts in ordered sets} 
Let $X$ be a totally ordered set. Recall that a \emph{cut} in $X$ is a pair $(L,R)$ of subsets of $X$ such that $L \cup R = X$ and $L < R$. The cuts $(\emptyset,X)$ and $(X,\emptyset)$ are respectively denoted by $-\infty$ and $+\infty$. If $Y \subseteq X$, then $Y^+$ denotes the cut $(L,R)$ with $R=\Set{x \in X \mid x > Y}$. Similarly, $Y^-$ is the cut $(L,R)$ with $L=\Set{x \in X \mid x < Y}$. For $a\in X$, we write $a^+$ for $\{a\}^+$, and $a^-$ for $\{a\}^-$.

\subsubsection*{Coloured Chains}\label{ColouredChains}

Coloured chains appear naturally in the study of ordered abelian groups, as certain chains of convex subgroups called \emph{spines} play an important role. We briefly discuss stably embedded pairs of coloured chains. 
 
\begin{definition}
A \textit{coloured chain} is a structure of the form $(C,<,(P_i)_{i \in I})$, where $(C,<)$ is a totally ordered set, and where each $P_i$, for $i\in I$, is a unary predicate.

\end{definition}

Let us observe that, as there is no reasonable language in which all coloured chains admit quantifier elimination, there is also a priori no good general characterisation of stably embedded coloured chains. But chains have been quite intensively studied. In \cite[Section~12.6]{Poi00}, Poizat has shown that a non-realised $1$-type $p(x)=\tp(a/C)$ over a coloured chain $C$ is definable if and only if the cut $(L_a,R_a)$ of $a$ over $C$, where $L_a:=\{c\in C\mid c<a\}$ and $R_a=\{c\in C\mid c>a\}$, is definable. This result and its proof extend to types in an arbitrary number of variables:

\begin{proposition}
Let $C\preccurlyeq D$ be coloured chains and $a_1,\ldots,a_n\in D$. Then $tp(\overline{a}/C)$  is definable if and only if $tp(a_i/C)$ is definable for all $i$ if and only if whenever $a_i\not\in C$ the cut of $a_i$ over $C$ is definable. 
\end{proposition}

%============================== Proof  ===================================
%\begin{proof}
%If $p$ is definable, then trivially $(A^p_i,B^p_i)$ is definable for every $1 \le i \le n$. Conversely, suppose that $(A^p_i,B^p_i)$ is definable for every $1 \le i \le n$. Let $D$ be a very saturated elementary extension of $C$. By Theorem \ref{TheoremLascarPoizat}, it is sufficient to show that $p$ has only one heir over $D$. In particular, we need to show that for an heir $q \in S_n(D)$ of $p$, there is for each $i$, $1 \le i \le n$, only one possibility for its cut $(A_i^q,B_i^q)$ and its side of satisfiability at $x_i$. Let $1 \le i \le n$ be fixed, and let $\psi(x_i)$ be a definition of $A_i^p$. If $q \in S_n(D)$ is an heir of $p$, then for every $d \in D$, $D \models \psi(d)$ if and only if $q \vdash d<x_i$. Therefore, $A_i^q=\Set{d \in D \mid d<x_i \in q}$ is necessarily the subset of $D$ defined by $\psi(x_i)$ and, hence, we have no choice for the cut determined by $q$ and $x_i$. The same is for its side of satisfiability. Suppose, for instance, that $p$ is satisfiable on the left at $x_i$. Then, $p$ is finitely satisfiable in $A_i^p$ and, so, $q$ is finitely satisfiable in $A_i^q$.  
%\end{proof}

We thus obtain the following characterisation of stably embedded coloured chains.

\begin{corollary}\label{FactStablyEmbeddedChainIFFAllCutDefinable}
A coloured chain $C$ is stably embedded in an elementary extension $D$ if and only if all cuts of $C$ realised in $D$ are definable.
In particular, a chain $C$ is stably embedded if and only if all cuts of $C$ are definable.
\end{corollary}

We will (abusively) denote by $\omega$ the chain $(\omega,<)$ and by $\omega^*$ the same underlying set equipped with the reversed order.

\begin{example}\label{ExampleStablyEmbeddedChains}
\begin{enumerate}
    \item The following total orders, expanded with arbitrary colours, are stably embedded:
    \begin{itemize}
        \item $\omega$, $\omega^*$, $(\mathbb{Z},<)$,
        \item $(\mathbb{R}, <)$.
      % \item $(\mathbb{N},<)+(\mathbb{Z},<)$
    \end{itemize}
    
    \item Let $C$ be an arbitrary coloured chain and let $(P_i,Q_i)_{i\in\lambda}$ be an enumeration of all non-principal\footnote{Recall that a \emph{prinicpal} is a cut of the form $a^+$ or $a^-$ for some element $a$.} cuts of $C$. Then, as all its cuts are trivially definable, $(C,(P_i)_{i\in\lambda})$ is a stably embedded coloured chain by Corollary \ref{FactStablyEmbeddedChainIFFAllCutDefinable}. For example,  the ordered sum $(C,<)=\omega+\omega^*$ is not stably embedded, but equipped with a predicate for $\omega$, it becomes stably embedded.
    \end{enumerate}
\end{example}

The study of stably embedded coloured chains could be pushed further. For instance, a characterisation of the class of stably embedded chains, given by an inductive construction in a similar way as in \cite[Definition~5.8]{Rub74}, would be useful. This question however will not be explored in this paper.

\subsection{Generalities on ordered abelian groups}\label{SS:GeneralitiesonOrderedAbelianGroups}
 All ordered abelian groups will be considered as structures in the language ${\mathcal{L}_{\text{oag}}:=\{0,+,-,<\}}$. 
We say that an ordered abelian group $G$ is \emph{discrete} if there exists a minimal positive element, and \emph{dense} otherwise. Note that the set of all convex subgroups of $G$ is linearly ordered by inclusion, and any intersection (and also any union) of convex subgroups is again a convex subgroup. Given $X$ a subset of $G$, we define the convex subgroup generated by $X$, denoted by $\Braket{X}^{\text{conv}}$, as the smallest convex subgroup containing $X$. In particular, we mean by a \emph{principal convex subgroup} $C$ one generated by a single element $a \in G$, i.e. $C=\Braket{a}^{\text{conv}}$. We denote by $\div(G)$ the divisible hull of $G$. 

\subsubsection*{Cuts in ordered abelian groups} 
Let $(L,R)$ be a cut in an ordered abelian group $G$. For any $g \in G$, we define $g+(L,R)$ as the cut $(g+L,g+R)$. Then, to every cut $(L,R)$ in $G$, we may associate the following convex subgroup 
\[
\Inv(L,R) := \Set{g \in G \mid g+(L,R)=(L,R)},
\]
of $G$, called the \emph{invariance group} of $(L,R)$. Note that, for every convex subgroup $C$ of $G$, we have $\Inv(g+C^+)=\Inv(C^+)=C=\Inv(C^-)=\Inv(g+C^-)$ for any $g \in G$.
\subsubsection*{Hahn product and lexicographic sum}
 Let $(I,<)$ be a chain, and for each $i \in I$ let $R_i$ be an ordered abelian group. We denote the direct product of the groups $R_i$ by $\prod_{i\in I} R_i$. For every $f\in \prod_{i\in I} R_i$, the \emph{support} of $f$ is the set $\supp(f):=\Set{i \in I \mid f(i) \neq 0}$.
The \emph{Hahn product} of $\{R_i\}_{i \in I}$,  denoted by $\hprod_{i\in I} R_i$ is the abelian group 
\begin{equation*}
\Set{f \in \prod_{i \in I}R_i \mid \supp(f) \text{ is a well-ordered subset of } (I,<)}.
\end{equation*}
equipped with the \emph{lexicographic order}. It is an ordered abelian group (i.e., the addition is compatible with the order) and the positive cone is given by 
    \[ f>_{lex}0 \iff f({\min(\supp(f))})>0.\]
The \emph{lexicographic sum} of $\{R_i\}_{i \in I}$, denoted by $\sum_{i \in I} R_i$, is the ordered subgroup 
\begin{equation*}
\Set{ f \in \hprod_{i \in I} R_i \mid \supp(f) \text{ is finite }}.
\end{equation*}
The lexicographic sum and the Hahn product are indistinguishable by first order properties:

\begin{proposition}[{\cite[Corollary 6.3]{Sch82}}]
\label{PropositionLexicographicSum}
Let $(I,<)$ be a linearly ordered set and for each $i \in I$ let $R_i$ be an ordered abelian group. Then $\sum_{i\in I}R_i\preccurlyeq\hprod_{i \in I}R_i$ as ordered abelian groups.
\end{proposition}

Lexicographic sums and Hahn products can be naturally equipped with a valuation map, namely $(a_i)_i \mapsto \min \{i  \mid a_i \neq 0\}$. 
Valuations on abelian groups plays an important role in our study, and we should define it now:

\begin{definition}\label{DefinitionValuedGroup}
A \emph{valued abelian group} is an abelian group $G$ equipped with a surjective map $val:G \rightarrow \Gamma$, where $\Gamma$ is a linearly ordered set with maximal element $\infty$, and where $val$ satisfies the following axioms:
\begin{enumerate}[label=(\roman*)]
\item\label{(i)} for all $g \in G$, $\val(g)=\infty \iff g=0$, 
\item\label{(ii)}  for all $g,h \in G$,
$\val(g-h) \ge \min \{\val(g),\val(h)\}$.
%\item[(iv)] $\val(a) \neq \val(b) \Rightarrow \val(a-b) = \min \{\val(a),\val(b)\}$.
\end{enumerate}

If \ref{(i)} is replaced by the condition $val(0)=\infty$, the map $val$ is called a \emph{pre-valuation} on $G$.

 A valued abelian group $(G,val)$ which satisfies moreover the following axiom will be called $\mathbb{Z}$-\emph{invariant}:
 \begin{enumerate}[label=(\roman*)]
 \setcounter{enumi}{2}
     \item\label{(iii)} for all $g\in G$ and $n\in \mathbb{N}_{>0}$, $\val(ng)=\val(g)$. 
 \end{enumerate}
\end{definition}

As an easy consequence of the axioms,  one sees that if $val$ is a pre-valuation on the abelian group $G$, then for any $g,h\in G$ one has  $\val(-g)=\val(g)$ and $\val(g-h) > \min \{\val(g),\val(h)\}$ only if $\val(h)=\val(g)$. 
We may also note that $\mathbb{Z}$-invariance implies torsion-freeness.

We will see in the next paragraph that all ordered abelian groups are naturally equipped with a $\mathbb{Z}$-invariant valuation.

\subsubsection*{The archidemean  skeleton}
We introduce a fundamental algebraic invariant of an ordered abelian group.

\begin{definition}\label{DefinitionArchimedeanSkeleton}
Let $G$ be an ordered abelian group. We denote by $\Gamma_G^a$ or simply $\Gamma^a$ a set indexing the set $\{\Braket{g}^{\text{conv}}\}_{g \in G}$ of principal convex subgroups of $G$, and reversely ordered with respect to the inclusion: for any $\gamma,\delta \in \Gamma_G^a$,
\begin{equation*}
\gamma<\delta \iff C_\delta^a \subsetneq C_\gamma^a,
\end{equation*}
where, for any $\gamma \in \Gamma_G^a$, $C_\gamma^a$ denotes the corresponding principal convex subgroup. Then, $\Gamma_G^a$ has a maximal element corresponding to $\{0\}$, which we denote by $\infty$. \\ For every $\gamma \in \Gamma_G^a \setminus \{\infty\}$, we denote by $V_\gamma^a$  the union of all convex subgroups strictly contained in $C_\gamma^a$. For $\gamma=\infty$, set $V_\gamma^a=\emptyset$. Note that for any $\gamma \in \Gamma_G^a\setminus \{\infty\}$, the quotient $R_\gamma^a=C_\gamma^a/V_\gamma^a$ is a non-trivial archimedean ordered abelian group, called the \textit{rib} of $\gamma$.

If $g \in G, g \neq 0$ such that $C_\gamma^a=\Braket{g}^{\text{conv}}$, we may also write $C^a_g$ (resp. $V^a_g$ or $R^{a}_g$) instead of $C^a_\gamma$ (resp. $V^a_\gamma$ or $R^{a}_\gamma$), and call it the \textit{cover} of $g$ (resp. the \textit{fundament} or the \textit{rib} of $g$). The cover $C^a_g$ (resp. the fundament $V^a_g$) is by definition the smallest convex subgroup containing $g$ (resp. the largest convex subgroup not containing $g$).

The pair
\begin{equation*}
(\Gamma_G^a, (R_\gamma^a)_{\gamma \in \Gamma_G^a\setminus \{\infty\}})
\end{equation*}
is called the archimedean \emph{skeleton} of $G$. Moreover, we call the set $\Gamma_G^a$ the \emph{archimedean spine} of $G$, a pair $(\gamma, R_\gamma^a)$ a \emph{bone} of $G$, and $R_\gamma^a$ a \emph{rib} of $G$, for any $\gamma \in \Gamma_G^a$. 
\end{definition}

\begin{example}Assume that $\{R_i\}_{i \in I}$ is a family of non-trivial archimedean ordered abelian groups.
The archimedean skeleton of the lexicographic sum $\sum_{i \in I}R_i$ and of the Hahn product $\hprod_{i \in I}R_i$ are both given by 
\[(I\cup \{\infty\},(R_i)_{i \in I}).\]
\end{example}

%Furthermore, the order structure on $G$ induces a metric structure, coming from the natural valuation, defined as follows.

\begin{definition}\label{DefinitionNaturalValuation}
Let $G$ be an ordered abelian group. The \emph{natural valuation} (also called \emph{archimedean} valuation) on $G$ is the map 
\begin{equation*}
\val^a \colon G \to \Gamma_G^a
\end{equation*}
defined by
\begin{equation*}
\val^a(g) = \gamma,\text{where } \Braket{g}^{\text{conv}}=C_\gamma^a.
\end{equation*}
\end{definition}

Note that, for every $g \in G$, $g \neq 0$, 
\[
C_g^a=\Set{h \in G \mid \val^a(h) \ge \val^a(g)}, \quad \text{and } V_g^a=\Set{h \in G \mid \val^a(h)>\val^a(g)}.
\]

The following fact is immediate: 

\begin{fact}Any ordered abelian group is a $\mathbb{Z}$-invariant valued group with respect to the natural valuation. 
\end{fact}

To close this paragraph, and for completeness, we state a fundamental result in the theory of ordered abelian groups, which highlights the important role of the skeleton and the Hahn product (see, e.g., the proof of \cite[Corollary~3.11]{SS95}).

\begin{fact}[Hahn Embedding Theorem]
Let  $G$ be an ordered abelian group with archimedean skeleton $(\Gamma_G^{a}, (R_\gamma^{a})_{\gamma \in \Gamma_G^{a}})$.
Then $G$ embeds, as an ordered abelian group, into $\hprod_{\gamma \in \Gamma_G^{a}\setminus\{\infty\}} \div(R_\gamma^{a})$.
\end{fact}

\subsubsection*{Regular ordered abelian groups}
In the language $\mathcal{L}_{\text{oag}}$, ordered abelian groups do usually not eliminate quantifiers.
It is a classical result that \emph{Presburger Arithmetic} PRES, i.e., the theory of the ordered abelian group of the integers $\Z$, admits quantifier elimination in the \emph{Presburger language} $\mathcal{L}_{\text{Pres}}:=\{0,1,+,-,<,(\equiv_m)_{m \in \mathbb{N}}\}$ (see for example \cite[Theorem 2.5]{Wei81}). Note that every ordered abelian group can be seen as an $\mathcal{L}_{\text{Pres}}$-structure: the symbols $0,+,-,<$ are interpreted in the obvious way; the constant symbol $1$ is interpreted by the least positive element if $G$ is discrete, and by $0$ otherwise; for each $m \in \mathbb{N}$, the binary relation symbol $\equiv_m$ is interpreted as congruence modulo $mG$.

Archimedean ordered abelian groups do not form an elementary class. Robinson and Zakon \cite{RoZa60,Zak61} have introduced and studied the class of regular ordered abelian groups which is the smallest elementary class containing all archimedean ordered abelian groups.

\begin{definition} Let $G$ be an ordered abelian group.
\begin{itemize}
        \item $G$ is called \emph{regular} if for any non-trivial convex subgroup $H$ the quotient group $G/H$ is divisible.%  It is called \emph{regular discrete} (resp. \emph{regular dense} if it is regular and discrete (resp. regular and dense).
        \item For $n>1$, $G$ is called \emph{$n$-regular} if any interval in $G$ containing at least $n$ elements contains an element divisible by $n$. 
    \end{itemize}
\end{definition}

\begin{fact}[\cite{RoZa60,Zak61,Con62,Wei86}]\label{F:RegChar}
For $G$ an ordered abelian group the following are equivalent:
\begin{enumerate}
\item $G$ is regular.
\item $G$ is $n$-regular for every positive integer $n>1$.
\item The theory of $G$ eliminates quantifiers in $\mathcal{L}_{\text{Pres}}$.
%\footnote{The constant $1$ is interpreted by the smallest element of $G$ if it exists, and by $0$ otherwise.}

\item $G$ is elementarily equivalent to some archimedean ordered abelian group.
\item The only definable convex subgroups of $G$ are $\{0\}$ and $G$.
\item For any prime number $p$ and any infinite convex subset $A\subseteq G$ there is an element $a\in A$ such that $a=p\cdot b$ for some $b\in G$.
\end{enumerate}
%If we assume in addition that $G$ is dense, then the above statements are moreover equivalent to the following:
%\begin{enumerate}[resume]
%\item $G$ is dense in its divisible hull.
%\end{enumerate}
\end{fact}

\begin{fact}[\cite{RoZa60}]\label{convex Sg. elt. equiv.}
    All regular discrete groups are models of Presburger arithmetic. 
    Two non-trivial regular dense groups $G$ and $H$ are elementary equivalent if and only if for any prime $p$ we have that both $|G/pG|$ and $|H/pH|$ are infinite or $|G/pG|=|H/pH|$ holds.
\end{fact}

\subsubsection*{Quantifier elimination}We have seen that (precisely) the regular ordered abelian groups admit quantifier elimination in $\mathcal{L}_{\text{Pres}}$. Roughly speaking, in order to eliminate quantifiers in the class of all ordered abelian groups, we need a (many-sorted) expansion of $\mathcal{L}_{\text{Pres}}$ that can deal with the ordered abelian groups $G/H$, where $H$ is a definable convex subgroup of $G$. To this end, we recall the language $\mathcal{L}_{\text{syn}}$ introduced by Cluckers and Halupczok \cite{CH11}. We begin by describing the set of \emph{auxiliary sorts} $\mathcal{A}:= \{\mathcal{S}_n,\mathcal{T}_n, \mathcal{T}_n^+ \mid n \in \mathbb{N}, n>0 \}$ of $\mathcal{L}_{\text{syn}}$.
\begin{definition} \label{DefinitionAuxiliarySorts}
Fix a natural number $n>0$.
\begin{enumerate}
    \item For $g \in G \setminus nG$, let $G_g^n$ be the largest convex subgroup $H$ of $G$ such that $g \notin H+nG$; for $g \in nG$, set $G_g^n=\{0\}$. Define $\mathcal{S}_n := G/ \sim$, with $g \sim g'$ if and only if $G_g^n=G_{g'}^n$, and let $\mathfrak{s}_n \colon G \twoheadrightarrow \mathcal{S}_n$ be the canonical map. Denote by $G_{\alpha}$ the convex subgroup $G_g^n$, with $\alpha=\mathfrak{s}_n(g)$.
    \item For $g \in G$, set $H_g^n=\bigcup_{h \in G, g \notin G_h^n}G_h^n$, where the union over the empty set is $\emptyset$
    \footnote{ With this convention, we have $H_g^n=\emptyset$ if $g\in nG$. In \cite{CH11}, we would have $H_g^n=\{0\}$ instead. We adopt this convention to avoid some case distinctions and turn $\mathfrak{t}_n$ into a pre-valuation. }.

    Define $\mathcal{T}_n := G/ \sim$, with $g \sim g'$ if and only if $H_g^n=H_{g'}^n$, and let $\mathfrak{t}_n \colon G \twoheadrightarrow \mathcal{T}_n$ be the canonical map. Denote by $G_{\alpha}$ the convex subgroup $H_g^n$, with $\alpha=\mathfrak{t}_n(g)$.
    \item Denote by $\mathcal{T}_n^+$ a copy of $\mathcal{T}_n$, i.e., $\mathcal{T}_n^+:=\{\beta^+\}_{\beta \in \mathcal{T}_n}$. For each $\beta^+ \in \mathcal{T}_n^+$, let $G_{\beta^+}=\bigcap_{\alpha \in \mathcal{S}_n, G_{\alpha} \supsetneq G_{\beta}} G_{\alpha}$, where the intersection over the empty set is $G$. In particular, if $\beta=\mathfrak{t}_n(g)$, we have $G_{\beta^+}=\bigcap_{h \in G, g \in G_h^n}G_h^n$.
\end{enumerate}
\end{definition}

In \cite{CH11}, it is proved that, for each $n>0$, the convex subgroups of the three families in the above definition $\{G_{\alpha}\}_{\alpha \in \mathcal{S}_n}, \{G_{\alpha}\}_{\alpha \in \mathcal{T}_n}$ and $ \{G_{\alpha}\}_{\alpha \in \mathcal{T}_n^+}$ are uniformly definable in  $\mathcal{L}_{\text{oag}}=\{0,+,-,<\}$. It follows that in any theory of ordered abelian groups, all the auxiliary sorts are imaginary sorts of $\mathcal{L}_{\text{oag}}$.

For any $\alpha \in \bigcup_{n \in \mathbb{N},n>0} \mathcal{S}_n\cup \mathcal{T}_n \cup \mathcal{T}_n^+$ and $m \in \mathbb{N}, m>0$, we also set
\begin{equation}
G_{\alpha}^{[m]} := \bigcap_{H \supsetneq G_\alpha, H \text{ convex subgroup of } G} (H + mG).
\end{equation}
Now we are able to present the complete definition of $\mathcal{L}_{\text{syn}}$. 
\begin{definition}\label{DefinitionLsyn}
The language $\mathcal{L}_{\text{syn}}$ consists of the following:
\begin{enumerate}[label=(\alph*)]
    \item\label{a} The main sort $(G,0,+,-,<,(\equiv_m)_{m \in \mathbb{N}})$;
    \item\label{b} the auxiliary sorts $\mathcal{S}_n,\mathcal{T}_n, \mathcal{T}_n^+$, for each $n \in \mathbb{N} \setminus \{0\}$, with the binary relations $\le$ on ($\mathcal{S}_n \cup \mathcal{T}_n \cup \mathcal{T}_n^+) \times (\mathcal{S}_m \cup \mathcal{T}_m \cup \mathcal{T}_m^+)$ (each pair $(m,n)$ giving rise to nine binary relations), defined by $\alpha \le \alpha'$ if and only if $G_{\alpha'} \subseteq G_{\alpha}$\footnote{In \cite{CH11}, the order is defined so that it corresponds to the inclusion of convex subgroups. Again, we use the reverse order so that $\mathfrak{t}_n$ is a pre-valuation.};
    \item\label{c} the canonical maps $\mathfrak{s}_n \colon G \twoheadrightarrow \mathcal{S}_n$ and $\mathfrak{t}_n \colon G \twoheadrightarrow \mathcal{T}_n$, for each $n \in \mathbb{N} \setminus \{0\}$;

    \item\label{d} a unary predicate $x=_\bullet k_\bullet$ on $G$, for each $k \in \Z \setminus \{0\}$, defined by $g=_\bullet k_\bullet$ if and only if there exists a convex subgroup $H$ of $G$ such that $G/H$ is discrete and $g \mod H$ is equal to $k$ times the smallest positive element of $G/H$, for every $g \in G$; 
    \item\label{e} a unary predicate $x\equiv_{\bullet m} k_\bullet$ on $G$, for each $m \in \mathbb{N} \setminus \{0\}$ and $k \in \{1,\dots,m-1\}$, defined by $g \equiv_{\bullet m} k_\bullet$ if and only if there exists a convex subgroup $H$ of $G$ such that $G/H$ is discrete and $g \mod H$ is congruent modulo $m$ to $k$ times the smallest positive element of $G/H$, for every $g \in G$;
    \item\label{f} a unary predicate $D_{p^r}^{[p^s]}(x)$ on $G$, for each prime $p$ and each $r,s \in \mathbb{N} \setminus \{0\}$ with $s \ge r$, defined by $D_{p^r}^{[p^s]}(g)$ if and only if there exists an $\alpha \in \mathcal{S}_p$ such that $g \in G_{\alpha}^{[p^s]}+p^rG$ and $g \notin G_{\alpha}+p^rG$, for every $g \in G$;
    \item \label{g}a unary predicate discr$(x)$ on the sort $\mathcal{S}_p$, with $p$ prime, defined by discr$(\alpha)$ if and only if $G/G_{\alpha}$ is discrete, for every $\alpha \in \mathcal{S}_p$;
    \item\label{h} two unary predicates on the sort $\mathcal{S}_p$, with $p$ prime, for each $l,n \in \mathbb{N} \setminus \{0\}$, defining the sets
    \begin{gather*}
    \{\alpha \in \mathcal{S}_p \mid \dim_{\mathbb{F}_p}(G_{\alpha}^{[p^n]}+pG)/(G_{\alpha}^{[p^{n+1}]}+pG)=l\} \text{ and} \\
    \{\alpha \in \mathcal{S}_p \mid \dim_{\mathbb{F}_p}(G_{\alpha}^{[p^n]}+pG)/(G_{\alpha}+pG)=l\}.
    \end{gather*}

\end{enumerate}   
\end{definition}

\begin{fact}[{\cite[Theorem 1.13]{CH11}}]\label{FactQUantifierEliminationLsyn}
The $\mathcal{L}_{syn}$-theory of ordered abelian groups eliminates quantifiers from the main sort $G$.
\end{fact}

    Notice that in \cite{CH11}, $\mathcal{L}_{syn}$ only includes the sort $\mathcal{S}_p,\mathcal{T}_p$ and $\mathcal{T}_p^+$ for $p$ running over the prime numbers. In particular, the statement of \cite[Theorem 1.13]{CH11} is slightly stronger than the result cited above, and we have that any $\mathcal{L}_{\text{syn}}$-formula $\phi(\bar{x},\bar{\eta})$, with $G$-variables $\bar{x}$ and $\mathcal{A}$-variables $\bar{\eta}$, is a boolean combination of formulas of the form
\begin{itemize}
	\item $\psi(\bar{x})$, where $\psi$ is quantifier free in the language $(G,0,+,-,<,(\equiv_m)_{m \in \mathbb{N}})$ of the main sort, and
	\item $\chi(\bar{x},\bar{\eta}):=\xi((\mathfrak{s}_{p^r}(\sum_{i<n}z_ix_i),\mathfrak{t}_p(\sum_{i<n}z_ix_i))_{p\in \mathbb{P},r\in \mathbb{N},z_0, \dots, z_{n-1} \in \Z},\bar{\eta})$, where $\xi$ is an $\mathcal{A}$-formula.
 %and $z_0, \dots, z_{n-1} \in \Z$.
\end{itemize}

The following fact will be useful as well:
\begin{fact}[{\cite[Lemma 2.12]{CH11}}]\label{FactQE}
For any $g \in G$, we have the following equivalences.
\begin{enumerate}
	\item \label{1} $g =_\bullet k_\bullet$ if and only if $G/G_{\mathfrak{t}_2(g)}$ is discrete and $g \mod G_{\mathfrak{t}_2(g)}$ is equal to $k$ times the smallest positive element of $G/G_{\mathfrak{t}_2(g)}$.
	\item $g \equiv_{\bullet m} k_\bullet$ if and only if $G/G_{\mathfrak{s}_m(g)}$ is discrete and $g \mod G_{\mathfrak{s}_m(g)}$ is congruent modulo $m$ to $k$ times the smallest positive element of $G/G_{\mathfrak{s}_m(g)}$.
\end{enumerate}
\end{fact}

As remarked in \cite{CH11}, the map $\mathfrak{t}_2$ can be replaced by any other map $\mathfrak{t}_p$, with $p \in \mathbb{P}$.

\subsubsection*{Induced structure on convex subgroups}
Let us finally mention another useful tool for the study of ordered abelian groups. First, we recall the notion of a pure exact sequence:

\begin{definition}
    An embedding $\iota: A\rightarrow B$ of abelian groups is said \emph{pure} if for all positive integers $n$ and all elements $a\in A$, $\iota(a)$ is divisible by $n$ in $B$ if and only if $a$ is divisible by $n$ in $A$. 

    We say by extension that the short exact sequence 
    \[0 \overset{\iota}{\rightarrow} A \rightarrow B \rightarrow C \rightarrow 0\]
    is pure if the embedding $\iota: A\rightarrow B$ is pure. 
\end{definition}
Notice that if $C$ is torsion-free, then automatically a short exact sequence $0 \rightarrow A \rightarrow B \rightarrow C \rightarrow 0$ is pure. 
To avoid any confusion with the model-theoretic property of purity, we will avoid saying that $A$ is a pure subgroup of $B$.

\begin{fact}\label{Fact:Purity-convex}
    Let $G$ be an ordered abelian group and let $H\leq G$ be a convex subgroup. Consider $G$ in the language $\mathcal{L}_{\text{oag}}\cup\{P\}$, where $P$ is a unary predicate singling out $H$. Then $H$ and $G/H$ are stably embedded in $G$, with induced structure that of a pure ordered abelian group. 
    
    In particular, this holds when $H$ is an $\mathcal{L}_{\text{oag}}$-definable convex subgroup of $G$.
\end{fact}

\begin{proof}
    If $G$ is an ordered abelian group and $H\leq G$ is a convex subgroup, then 
one obtains the pure short exact sequence $$0\rightarrow H\rightarrow G\rightarrow G/H\rightarrow0$$
of abelian groups. %Replacing $G$ by an elementary extension if necessary, one may assume that $G$, and thus also the ordered abelian group $H$, is $\aleph_1$-saturated. Then $H$ is pure-injective (see, e.g., \cite[Theorem~27]{Che76}), so the above short exact sequence is split, and thus $G\cong G/H\times H$, i.e., $G$ is isomorphic to the lexicographic product of $G/H$ and $H$. 

%Since $H$ and $G/H$ are purely stably embedded ordered abelian groups in the product structure, this still holds in the $\mathcal{L}_{\text{oag}}\cup\{P\}$- structure $G$, which is a reduct of the product structure.

We note that the ordering on $G$ may be defined using the orderings on $H$ and on $G/H$, and thus the structure can be seen as a pure short exact sequence of abelian groups $0 \rightarrow A \rightarrow B \rightarrow C \rightarrow 0$, where $A \coloneqq H$ and $C \coloneqq G/H$ are enriched with orderings. The statement then follows directly from \cite[Remark~4.4]{ACGZ20}.% for short exact sequences $0\rightarrow A\rightarrow B\rightarrow C\rightarrow 0$ of abelian groups, with $A$ pure in $B$ and possibly enriched on $A$ and on $C$.
\end{proof}

\subsection{Valued abelian groups}\label{SS:KaplanskyTheory}
In this section, we develop some tools for valued abelian groups. In particular, we will need the notion of \emph{maximality} from general valuation theory, and the characterisation of this notion in terms of pseudo-completeness (Theorem \ref{TheoremMaximalityEquivalentPseudo-completeness}). This relationship between these two properties was originally proved by Kaplansky in the context of valued fields \cite{Kap42}. His ideas have been adapted to other contexts; see e.g. \cite{Grav55} for the case of valued vector spaces and \cite{Hol63} for the case of ordered groups, even non-commutative. We present here a Kaplansky theory for extensions $H/G$  of valued abelian groups.  %Interested reader can also refer to \cite{SS91}.

Following the usual convention, we call "convex subgroup" a subgroup of a valued group $G$ which is the inverse image of an end segment of $\Gamma$, even if $G$ is not equipped with a compatible order. We also extend the other notions introduced for ordered abelian groups: 
\begin{definition}\label{DefinitionSkeleton}
If $g$ is an element of value $\gamma$, we call \emph{fundament of $g$} the largest convex subgroup not containing $g$, i.e., 
\[V_\gamma\coloneqq V_g \coloneqq \{x\in G \ \vert \ \val(x)> \gamma \},\]

and \emph{cover of $g$} the smallest convex subgroup containing $g$, i.e.,
\[C_\gamma\coloneqq C_g\coloneqq \{x\in G \ \vert \ \val(x)\geq \gamma \}. \]
If $\gamma<\infty$, the \emph{rib of $g$} is defined to be the quotient
\[R_\gamma \coloneqq R_g \coloneqq C_\gamma/V_\gamma.\]

     We call 
    \[(\Gamma, (R_\gamma)_{\gamma \in \Gamma})\]
    the \emph{skeleton} of $(G,\val)$, and for $a\in G$ of value $\gamma$, we call
     a pair \[(\gamma, R_\gamma)\]
     the \emph{bone at $a$} or the \emph{bone at $\gamma$} of $G$. 
     \end{definition}

Firstly, notice that an embedding of valued abelian groups induces an embedding between their skeleta. Particularly noticeable is the case where the skeleta are actually equal:

\begin{definition}
\label{DefinitionMaximalOrderedAbelianGroups}
Let $G$ be a valued abelian group.
\begin{itemize}
    \item An extension $H\supseteq G$ of valued abelian groups is called \textit{immediate} if it preserves the skeleton, i.e., if $\Gamma_G=\Gamma_H$ and for each value $\gamma \in \Gamma_G$, $R_{G,\gamma}= R_{H,\gamma}$. 
    \item $G$ is called \emph{maximal} if it admits no proper immediate extension.
\end{itemize}
\end{definition}

\begin{example}
Let $I$ be a linearly ordered set and $\{G_i\}_{i \in I}$ a family of archimedean ordered abelian groups. Then,  working with the natural valuation (see Definition \ref{DefinitionNaturalValuation}), the extension of valued abelian groups $\hprod_{i \in I}G_i\supseteq\sum_{i \in I}G_i$ is immediate.
\end{example}

%\begin{remark}
%Not every immediate extension of an ordered abelian group $G$ is isomorphic to $G$, neither it is an elementary extension of $G$. Indeed, consider for instance the lexicographic sum $G= \bigoplus_{i<\omega} \mathbb{Q}$, and the element $a \in \hprod_{i<\omega} \mathbb{Q}$ such that $a(i)=1$ for every $i < \omega$. Then, the ordered group $G'=\Braket{G,a}$ generated by $G$ and $a$ is an immediate extension of $G$, but it is not elementarily equivalent: $a$ is not divisible by $2$ (or by any integer).
%\end{remark}

Our treatment of maximal valued abelian groups follows closely that of Kaplansky for maximal valued fields (\cite{Kap42}). We will also be interested in a relative notion of maximality, and will introduce, therefore, a Kaplansky theory for extensions of valued groups.

\begin{definition}
We say that an extension $H/G$ of valued abelian groups is \emph{maximal} if there is no non-trivial intermediate immediate extension, i.e., if for every valued abelian group $K$ with $G \subsetneq K \subseteq H$, either $\Gamma_G \subsetneq \Gamma_K$ or $R_{G,\gamma} \subsetneq R_{K,\gamma}$ for some $\gamma \in \Gamma_G$. 
\end{definition}

We have the following characterisation of immediate extensions.
\begin{fact}
\label{FactCharacterisationImmediateExtension} 
For an extension $H/G$ of valued abelian groups, the following are equivalent:
\begin{enumerate}
\item $H$ is an immediate extension of $G$,
\item for every $h \in H \setminus G$, there exists $g \in G$ such that $\val(h-g)>\val(h)$,
\item for every $h \in H \setminus G$, the set $\Delta=\Set{ \val(h-g) \mid g \in G}$ does not admit a maximum.
%\item for every $0<h \in H$ there exists $g \in G$ such that for all integers $n$, $n(h-g)<h$.
\end{enumerate}
\end{fact}
\begin{proof}
It is easy to show the equivalence $(2) \Leftrightarrow (3)$. We prove $(1) \Leftrightarrow (2)$. 

$(1) \Rightarrow (2)$: Let $H$ be an immediate extension of $G$ and $h \in H\setminus G$. Then, there is some $g \in G$ such that $\val(h)=\val(g)$ and $h  \mod V_g=g \mod V_g$. It follows that $\val(h-g)>\val(h)$. 

$(2) \Rightarrow (1)$: Let $h \in H\setminus G$. By (2) we find $g \in G$ such that $\val(h-g)>\val(h)$. Then, $\val(h)=\val(g)$ and $h \mod V_{g}=g \mod V_{g}$, so the bone of $h$ is equal to the the bone of $g$, proving that $H/G$ is immediate.
\end{proof}

%Notice that an extension $H$ of $G$ satisfying condition $4.$ in the above Fact is called a 'c-extension' in \cite{Hol63}.
To characterise maximal pairs of valued abelian groups, we need the following definitions.

\begin{definition}
Let $(G,\val)$ be a valued abelian group, and let $(a_i)_{i\in I}$ be a sequence in $G$, with $I$ a well-ordered index set with no maximal element. 
    \begin{itemize}
        \item  The sequence $(a_i)_{i \in I}$ is called  \textit{pseudo-Cauchy} if there is $\alpha \in I$ such that for all $\alpha< i<j<k$, $\val(a_i-a_j) < \val(a_j -a_k)$.
        \item If $(a_i)_{i \in I}$ is a pseudo-Cauchy sequence, an element $a\in G$ is called a \emph{pseudo-limit} of $(a_i)_{i \in I}$ if there exists $\alpha \in I$ such that for all $\alpha< i<j$, $\val(a_i-a)=\val(a_i-a_j)$.
        \item $(G,\val)$ is called \emph{pseudo-complete} if every pseudo-Cauchy sequence in $G$ admits a pseudo-limit in $G$.
    \end{itemize}
\end{definition}

\begin{remark}\label{rmk:EventuallyCstorIncreasing}
    If $(a_i)_{i\in I}$ is a pseudo-Cauchy sequence in some valued group $(G,\val)$, then $(\val(a_i))_i$ is eventually strictly increasing or eventually constant.  
\end{remark}
As for the notion of maximality, we relativise the notion of pseudo-completeness to an extension $H/G$ of valued abelian groups.
\begin{definition}
We say that $H/G$ is \emph{pseudo-complete} if every pseudo-Cauchy sequence in $G$ admitting a pseudo-limit in $H$ admits a pseudo-limit in $G$.
\end{definition}

Throughout the paper, if $(a_i)_{i \in I}$ is a sequence of elements in a valued group $G$, we will say that an assertion about its elements holds for $a_i$ \emph{eventually} if there is some $i_0 \in I$ such that it holds for all $a_i$ with $i \ge i_0$.

\begin{proposition}
\label{PropositionPseudo-completenessImpliesMaximality}
Let $K/G$ be an immediate extension of valued abelian groups, and let $h \in K \setminus G$. Then there is a pseudo-Cauchy sequence in $G$ with pseudo-limit $h$ and with no pseudo-limit in $G$. In particular, if $H/G$ is pseudo-complete, then $H/G$ is maximal; and if $G$ is pseudo-complete, then $G$ is maximal.
\end{proposition}

\begin{proof}
By Fact \ref{FactCharacterisationImmediateExtension}, the set $\Delta=\{\val(h-g) \ \vert \ g\in G\}$ does not admit a maximum. Consider a sequence $(a_i)_{i\in I}$ of elements in $G$ such that $\val(a_i-h)$ is strictly increasing and  cofinal in  $\Delta$. It follows that, for any $i<j<k$, $\val(a_i-a_j)=\min \{\val(a_i-h),\val(a_j-h)\}=\val(a_i-h) < \val(a_j-h)=\min \{\val(a_j-h),\val(a_k-h)\}=\val(a_j-a_k)$. Hence, the sequence $(a_i)_{i \in I}$ is pseudo-Cauchy. By construction $h$ is a pseudo-limit of $(a_i)_{i \in I}$, and $(a_i)_{i \in I}$ does not admit a pseudo-limit in $G$.
\end{proof}

We will prove the converse, but we need first the following lemma:

\begin{lemma}\label{LemmaMaximalImpliesDeltaMaximum}
    Let $(G,\val)$ be a $\mathbb{Z}$-invariant maximal valued abelian group. Then for any $g\in G$ and $m \in \mathbb{N}$, the set  $\Delta_m(g):=\{\val(g-mg')\mid g'\in G\}$ admits a maximum. In other words, the extension $G\supseteq mG$ of valued abelian groups is pseudo-complete.
\end{lemma}

\begin{proof}
The cases $m=0$ and $m=1$ trivially hold. Assume now that $m\geq2$. As $(G,val)$ is $\Z$-invariant, multiplication by $m$ induces an isomorphism 
of valued groups $$\rho_m:(G,val)\cong(mG,val)$$ which induces the identity on the value set $\Gamma$. In particular $(mG,val)$ is maximal, since $(G,val)$ is maximal by assumption.

Let us now first consider the case where $m=p$ is a prime number, and let $g\in G\setminus pG$. 

\begin{claim}
    For any $g'\in \langle pG\cup\{g\}\rangle\setminus pG$ one has $\Delta_p(g')=\Delta_p(g)$.
\end{claim}

\begin{proof}[Proof of the claim] We note first that we have:
\[ \langle pG\cup\{g\}\rangle\setminus pG=\bigcup_{0<i<p}(pG+ig).\] Since $p$ is prime, one has ${g\in \bigcup_{0<i<p}(pG+ig')}$ for any such $g'$, so by symmetry it is enough to show that $\Delta_p(g)\subseteq \Delta_p(g')$. Let $0<i<p$ and $h'\in G$ such that $g'=ig+ph'$. For any $h\in G$ we have $val(g-ph)=val(ig-iph)=val(g'-p(ih+h'))$. This shows that any $\gamma=val(g-ph)\in\Delta_p(g)$ lies in $\Delta_p(g')$, proving the claim.
\end{proof}

It follows in particular from the claim that if for $g\in G\setminus pG$ the set $\Delta_p(g)$ has no maximum, the same holds for any $g'\in \langle pG\cup\{g\}\rangle\setminus pG$, and so 
$\langle pG\cup\{g\}\rangle\supseteq pG$ is an immediate extension by Fact~\ref{FactCharacterisationImmediateExtension}. This contradicts maximality of $pG$, thus proving the case $m=p$.

The general case now follows by induction on $m$. Indeed, if $m=pm'$ for some prime number $p$, then inductively the extension $G\supseteq m'G$ is pseudo-complete, as is the extension $m'G\supseteq pm'G=mG$ by the prime number case. The result follows, since if $G/H$ and $H/I$ are pseudo-complete extensions of valued abelian groups, then $G/I$ is pseudo-complete as well.
\end{proof}

\begin{remark}
    Since the conclusion of Lemma~\ref{LemmaMaximalImpliesDeltaMaximum} is first order, we deduce that it holds in any $\Z$-invariant valued abelian group which is elementarily equivalent to a maximal one. 
\end{remark}

\begin{theorem}
\label{TheoremMaximalityEquivalentPseudo-completeness}
Let $G$ be a $\mathbb{Z}$-invariant valued abelian group.
Assume that for all $g\in G$ and $m\in \mathbb{N}$, the set $\Delta_m(g)=\{\val(g-mg')\mid g'\in G\}$ admits a maximum. Let $H/G$ be an extension of valued abelian groups, where $H$ is not necessarily $\mathbb{Z}$-invariant.
Then $H/G$ is maximal if and only if $H/G$ is pseudo-complete.

In particular, a $\mathbb{Z}$-invariant valued group $G$ is maximal if and only if it is pseudo-complete.
\end{theorem}
The condition on the sets $\Delta_m(g)$ will be later denoted by (M).  Notice that the characterisation of maximal $\mathbb{Z}$-invariant valued groups is a generalisation of \cite[Theorem C6]{Hol63}.% from ordered abelian groups to extensions of $\mathbb{Z}$-invariant valued abelian groups.

\begin{proof}
We start with a proof of the first statement: 

        ($\Leftarrow$) This direction follows immediately from Proposition~\ref{PropositionPseudo-completenessImpliesMaximality}.
        
        ($\Rightarrow$) Let $(a_i)_{i\in I}$ be a pseudo-Cauchy sequence in $G$ without pseudo-limit in $G$, and assume that $(a_i)_{i\in I}$ has a pseudo-limit $a$ in $H \setminus G$. Let us show that the valued abelian group $G'$ generated by $a$ and $G$ is an immediate extension of $G$.  
        \begin{claim}
            For all $g\in G$ and $m\in \mathbb{N}$, $(\val(ma_i-g))_{i\in I}$ is eventually constant.
        \end{claim}
        
          \begin{proof}[Proof of the claim]
        The case $m=0$ is trivial. Let $m$ be a positive integer and $g \in G$ be fixed. As $\val(mb)=\val(b)$ for all $b\in G$, $(ma_i)_{i\in I}$ is a pseudo-Cauchy sequence. Thus, $(ma_i-g)_{i\in I}$ is a pseudo-Cauchy sequence, and so in particular $(\val(ma_i-g))_{i\in I}$ is eventually constant or eventually strictly increasing. \\
        Suppose for contradiction that $(\val(ma_i-g))_{i\in I}$  is eventually strictly increasing. We may assume that it is strictly increasing. By assumption, there is $g'\in G$ such that $\val(g-mg')= \max \Delta_m(g)$. Then for all $i\in I$, $\val(a_i-g')=\val(ma_i-mg') \geq \min (\val(ma_i-g),\val(g-mg'))$. Since $(\val(ma_i-g))_{i\in I}$ is strictly increasing and $\val(ma_i-g) \leq \val(g-mg')$, it follows that $\val(a_i-g')$ is strictly increasing, and so $g'$ is a pseudo-limit in $G$ of $(a_i)_{i\in I}$. Contradiction.
        \end{proof}
        
Fix $m \in \mathbb{N}$ and $g\in G$. We have to show that there exists an element in $G$ with the same bone as $ma-g$. We claim that $ma-g$ is a pseudo-limit of the pseudo-Cauchy sequence $(ma_i-g)_{i\in I}$. Indeed, else $(\val(ma_i-ma))_{i\in I}$ would be eventually constant, say eventually equal to $\delta$, with $\delta<\val(ma_i-ma_j)=\val(a_i-a_j)=:\gamma_i$ for all $i<j$ large enough. On the other hand, one shows by induction on $m$ that  $\val(ma_i-ma)\geq\val(a_i-a)$, which equals $\gamma_i$ for $i$ large enough, a contradiction.

Let $\gamma \in \Gamma_G$ and $i_0 \in I$ be such that $\gamma=\val(ma_i-g)$ for every $i \geq  i_0$. These exist by the claim. It follows that $\val(ma-g)= \gamma$, so in particular $\val(ma-g) \in \Gamma_G$, and $ma-g \mod V_{ma-g} =ma-g \mod V_{ma_{i_0}-g} = ma_{i_0}-g \mod V_{ma_{i_0}-g}$.   

Let us now show the second statement. The direction ($\Leftarrow$) holds by Proposition~\ref{PropositionPseudo-completenessImpliesMaximality}. As for ($\Rightarrow$), by Lemma~\ref{LemmaMaximalImpliesDeltaMaximum} we may apply the first part to any maximal $\Z$-invariant $G$. We may thus finish the proof by noticing that whenever $(a_i)_{i<\lambda}$ is a pseudo-Cauchy sequence in a valued abelian group without pseudo-limit in $G$, by compactness there is an (elementary) extension $\widetilde{G}$ of $G$ containing a pseudo-limit of $(a_i)_{i<\lambda}$, and so if $G$ is not pseudo-complete, there is an extension $\tilde{G}/G$ which is not pseudo-complete.
%       Clearly, if $g \in nG$, the statement is trivial. Then, suppose $g \notin nG$, and, for a contradiction, assume that $(\val(na_i-g))_{i\in I}$ is not eventually constant. In particular, we may assume that $(\val(na_i-g))_{i\in I}$ is strictly increasing. Note that it follows that the corresponding sequence $(V_{na_i-g})_{i \in I}$ of convex subgroups is strictly decreasing. Let $V$ be the largest convex subgroup of $G$ which does not contain $na_i-g$ for all $i\in I$, i.e. \[V= \bigcap_{i\in I}V_{na_i-g}.\] 
 %       For any $a \in G \setminus nG$, let $V^n_a$ denote the largest convex subgroup of $G$ such that $a \notin V_a^n+nG$, as in Definition \ref{DefinitionAuxiliarySorts}. We have that $V^n_g=V^n_{g-na_i} \subset V_{g-na_i}$ for each $i \in I$. Hence, $V^n_g \subset V$, and $g \in V+nG$. Let $g'$ be such that $g-ng' \in V$. This means that $\val(g-ng')>\val(g-na_i)$ for every $i \in I$. Let $i,j \in I$ be such that $i<j$. Then, $\val(a_i-g')=\min \{\val(g-ng'),\val(g-na_i)\}=\val(na_i-g)=\val(a_i-a_j)$. Therefore, $g'$ is a pseudo-limit of $(a_i)_{i\in I}$. We get a contradiction since, by hypothesis, $(a_i)_{i\in I}$ does not admit a pseudo-limit in $G$.
    \end{proof}

Important examples of maximal valued abelian groups are Hahn products of archimedean ordered abelian groups:

\begin{proposition}[{\cite[Lemma C4]{Hol63}}]
\label{PropositionMaximalityofHahnproducts}
Let $I$ be a totally ordered set, and let $(G_i)_{i \in I}$ be a family of archimedian ordered abelian groups. Then the Hahn product $\hprod_{i \in I}G_i$ is pseudo-complete, and thus maximal, with respect to the natural valuation. 
\end{proposition}

%Not complete \begin{proof}
%        We sketch a proof using a diagonal argument. Consider a Cauchy sequence $(a^i)_{i\in I}$ with $a^i: \Gamma \rightarrow \bigcup_\gamma G_\gamma$. Define an element $a:\Gamma \rightarrow \bigcup_\gamma G_\gamma$ as follow :
        %\begin{itemize}
         %   \item If there is $i_0 \in I$ such that $\gamma> \val(a^i-a^j)$ for all $i,j \geq i_0$, set $a(\gamma)= a^{i_0}_\gamma$.
         %   \item If there is no such $\alpha$, set $a(\gamma)= 0 \in G_\gamma$.
        %\end{itemize}We see that this definition makes sense, and that we obtained a pseudo-limit in  $\hprod_{\gamma \in \Gamma}(G_{\gamma})$ of $(a^i)_i$.    \end{proof}

\subsection{The induced valued group modulo $m$}\label{SS:ValueGroupModm}

Let $G$ be a valued abelian group. We will show that, for each natural number $m$, the quotient $G/mG$ may naturally be endowed with the structure of a valued group. Let $m \in \mathbb{N}$. Similarly to ordered abelian groups (Definition~\ref{DefinitionAuxiliarySorts}), we define for $a \in G \setminus mG$ the set $V_a^m$ to be the largest convex subgroup $H$ of $G$ such that $a \notin H+mG$, i.e., 
\[V_a^m := \{x\in G \ \vert \ \forall g\in G \ \val(a+mg) < \val(x)\}.\]

For $a \in mG$, we set $V_a^m=\emptyset$. 

\begin{definition}\label{DefinitionmValuation}
Let $m\in \mathbb{N}$. We denote by $\Gamma^m_G$ or simply $\Gamma^m$ the quotient of $G$ modulo the equivalence relation $\sim^m $ given by 
\[\forall a,b,\in G \quad a\sim^m b \quad \Leftrightarrow \quad V_a^m=V_b^m.\]

We denote by $\val^m \colon G \to \Gamma^m$ the natural projection: for $g\in G$,
\[
\val^m(g)= [g]_{\sim^m} = \{ h\in G \ \vert \ h \sim^m g \}.
\]

We abusively denote by $\infty$ (without specifying $m$) the maximal element of $\Gamma^m$, corresponding to $\{h\in G \ \vert \ h \sim^m 0 \}=mG$.

We see $\Gamma^m$ as the set indexing the convex subgroups $\{ V^m_g\}_{g \in G}$ and write $ V^m_\gamma$ if $\gamma=\val^m(g)$.

Finally, we order $\Gamma^m$ with the reverse inclusion: for any $\gamma,\delta \in \Gamma^m$
\[ \gamma < \delta  :\iff V^m_\delta \subsetneq V^m_\gamma.\]
\end{definition}

\begin{lemma}
For any $m\in\mathbb{N}$, $\val^m$ defines a pre-valuation on $G$. Moreover, $\val^m$ induces a valuation $\widehat{\val}^m:G/mG \rightarrow \Gamma^m$ on the abelian group $G/mG$ via $\widehat{\val}^m(g \mod mG):=\val^m(g)$. 
\end{lemma}

\begin{proof}
   %This is a straight forward verification. 
   We check that $\widehat{\val}^m$ satisfies (i) and (ii).
   \begin{itemize}
       \item[(i)] For $g\in G$, we have $g\sim^m 0$ if and only if $g \in mG$. This means that $\widehat{\val}^m(g \mod mG)=\infty$ if and only if $g=0 \mod mG$, as required.
       \item[(ii)] For all $a,b,x\in G$, if $\val(x)> \val(a+b+mg)$ for all $g\in G$ then in particular  $\val(x)> \min(\val(a+mg),\val(b+mg'))$ for all $g,g'\in G$. This means that $\val(x)> \val(a+mg)$ for all $g\in G$ or  $\val(x)> \val(b+mg')$ for all $g'\in G$. This shows that $V_{a+b}^m\subseteq V_{a}^m\cup V_b^m$, showing that $\widehat{\val}^m(a+b)\geq \min(\widehat{\val}^m(a),\widehat{\val}^m(b))$.
   \end{itemize}

\end{proof}
   Note that, since for all $g\in G$ one has $\widehat{\val}^m(mg)=\infty$, $(G/mG, \widehat{\val}^m)$ is not $\mathbb{Z}$-invariant, unless $G$ is divisible by $m$ (in which case $G/mG$ is trivial).
We call the valued abelian group $(G/mG,\widehat{\val}^m)$ the \textit{induced valued group modulo $m$} of $G$ and $\Gamma^m_G$ the \textit{$m$-value set} of $G$, or the \emph{$m$-spine} of $G$. By abuse of notation, we will also sometimes write $\val^m$ to refer to the induced valuation $\widehat{\val}^m$.

\begin{remark}
    \begin{itemize}
        \item Notice that, if $G$ is an ordered abelian group and $\val^{a}$ the natural valuation, with $\Gamma=\Gamma^{a}$, then for any $m>0$, $\Gamma^m$ is exactly the auxiliary sort $S_m$ in Definition~\ref{DefinitionAuxiliarySorts}, and $\val^m$ corresponds to the canonical map $\mathfrak{s}_m$. In particular, the induced valued group modulo $m$ is interpretable in the ordered abelian group $G$. (Observe, though, that the archimedean valuation is \emph{not} always interpretable.)
        \item(Zerology) If $(G,\val)$ is a valued abelian group with value set $\Gamma$, then $\Gamma^0=\Gamma$ and $\val^0=\val$. For $m=1$, $\Gamma^1=\{\infty\}$ and $\val^1$ is the constant map equal to $\infty$.
    \end{itemize}
\end{remark}

The aim of this section is to prove that the property of pseudo-completeness transfers from $(G,\val)$ to the induced valued group modulo $m$, for all $m\in\mathbb{N}$. We first show that, in the case of a pseudo-complete $\mathbb{Z}$-invariant valued abelian group $(G,\val)$, we can actually identify $\Gamma^m$ with a subset of $\Gamma_G$, for any $m \in \mathbb{N}$.

\begin{definition}\label{Definition(M)}
    Let $(G,\val)$ be a valued abelian group. We denote by (M) the following axiom scheme:
    
\begin{align}\label{ConditionMaximality}
    \tag{M} \forall \ x\in G, \exists \ y\in G \ x\equiv_m y \wedge V_x^m=V_y, \text{ for all $m\in\mathbb{N}_{>0}$}.
\end{align}
\end{definition}

Notice that (M) is given by an infinite set of first order axioms in the language of valued groups. 
\begin{lemma}\label{LemmaValueSetModm}
Any pseudo-complete $\mathbb{Z}$-invariant valued abelian group $(G,\val)$ satisfies (M).
\end{lemma}

%\martin{We should check whether $(iii)'$ suffices as assumption here. See above.}
\begin{proof}
This is a corollary of Lemma~\ref{LemmaMaximalImpliesDeltaMaximum}. Let $m \in \mathbb{N}_{>0}$ and let $g \in G$ be fixed. If $g\in mG$, then $V^m_g=V_0=\emptyset$. Assume $g\in G \setminus mG$. We need to show that there exists $g'\in G$ such that $g \equiv_m g'$ and $V^m_g=V_{g'}$.
 By Lemma~\ref{LemmaMaximalImpliesDeltaMaximum}, the set of values $\Delta_m(g)$ has a maximum  $\val(g-ma')$ where $a' \in G$. Then we have 
 $V^m_g=V^m_{g-ma'}=V_{g-ma'}$.
\end{proof}

We will discuss again the link between the property (M) and maximality by giving a converse to Lemma~\ref{LemmaValueSetModm} for certain ordered abelian groups (see Corollary~\ref{CorollaryGroupEquivalentMaximalGroup}).  

If $(G,\val)$ satisfies (M), then for any $m \in \mathbb{N}$ we can see $\Gamma^m $ as a subset of $\Gamma$  and the $m$-valuation map $\val^m \colon G \to \Gamma^m$ is given by
\begin{equation}
\label{inducedvaluation}
\val^m(g)=
\begin{cases}
\min \{\gamma \in \Gamma \mid g \notin V_\gamma+mG\} & \text{ if } g \notin mG, \\
\infty & \text{ otherwise.}
\end{cases}
\end{equation}
Note that $\Gamma^m$ may be a proper subset of $\Gamma$. In particular, we easily infer the following result, which interests us mainly when we are dealing with the natural valuation in an ordered abelian group.

\begin{corollary}\label{CorValueSetmodm}
Let $(G,\val)$ be a $\mathbb{Z}$-invariant valued abelian group satisfying (M), and let $m \in \mathbb{N}_{>0}$. We have:
\[\Gamma^m\setminus\{\infty\}=\Set{\gamma \in \Gamma \mid R_{\gamma} \text{ is not divisible by }m}.\]
\end{corollary}

\begin{proof}
Let $\gamma \in \Gamma$ be a value for which the corresponding rib $R_\gamma$ is not divisible by $m$. This means that $\gamma<\infty$ and there exists $g \in C_\gamma$ such that for all $g' \in C_\gamma$, $\val(g-mg')=\gamma$ 
%(otherwise, $\val(g-mg')>\gamma$ and, $g+V_\gamma = mg'+V_\gamma$ for some $g' \in C_\gamma$)
. Thus, $V_\gamma$ is the largest convex subgroup of $G$ not intersecting $g+mG$ and, hence, $\gamma \in \Gamma^m$. Conversely, if $\gamma \in \Gamma^m\setminus\{\infty\}$, then there exists $g\in G$ such that $\val^m(g)=\gamma$. By assumption (M), we may assume that $\val(g)=\gamma$. Thus $g \in C_\gamma$ and $g \notin V_\gamma + mG$, so $R_\gamma$ is not divisible by $m$.
\end{proof}

\begin{example}
Consider the ordered abelian group $ G = \hprod_{p \in \mathbb{P}} \mathbb{Z}_{(p)}$ (where $\mathbb{Z}_{(p)}$ is equipped with the Euclidean ordering). Let $\val^a$ be the natural archimedean valuation on $G$. By Lemma~\ref{LemmaValueSetModm}, $(G,\val^a)$ satisfies (M). For $p \in \mathbb{P}$, we have $p\Z_{(p)} \subsetneq \Z_{(p)}$ and $q\Z_{(p)} = \Z_{(p)}$ for all primes $q \neq p$. It follows that $\Gamma^p=\{p,\infty\}$ and $\Gamma=\mathbb{P}\cup\{\infty\}=\bigcup_{p \in \mathbb{P}}\Gamma^p$. In particular, all convex subgroups are definable.
\end{example}

Note that, if $G$ is a valued abelian group which is not pseudo-complete, the convex subgroup $V^m_g$ is not necessarily equal to a fundament $V_\gamma$ for some $\gamma \in \Gamma$. Consider, for instance, the lexicographic sum $G':= \sum_{r \in \Q} \Z$, an increasing sequence of positive rationals $(r_k)_{k\in \mathbb{N}}$ converging to $\sqrt{2}$ and $a:= (a_r)_{r\in \mathbb{Q}} \in \hprod_{r \in \mathbb{Q}} \Z$ with 
\[a_r= \begin{cases} m & \text{ if } r=r_k \text{ for some }k\in \mathbb{N},\\
0 &\text{ otherwise}.
\end{cases}
\]
Denote by $G$  the ordered abelian group $\langle G' \cup \{a\}\rangle$ endowed with the natural valuation. In particular, $a$ is not divisible by $m$ in $G$ and $V^m_a= \sum_{i>\sqrt{2}} \Z$, so clearly $V^m_a \neq V_g$ for any $g\in G$.
    
\begin{proposition}\label{PropPseudoCompleteModm}
Let $(G,\val)$ be a $\mathbb{Z}$-invariant pseudo-complete valued abelian group. Then, for every $m \in \mathbb{N}_{>0}$, the induced valued group modulo $m$ $(G/mG, \val^m)$ is pseudo-complete.
\end{proposition}

\begin{proof}
Recall that by Lemma~\ref{LemmaValueSetModm}, $G$ satisfies $(M)$: for all $a\in G$, there exists $b_a \in G$ such that $a \equiv_m b_a$ and $\val^m(a)=\val(b_a)$. Let us show that a pseudo-Cauchy sequence $(a_i \mod mG)_{i\in I}$ in $G/mG$ can be lifted to a pseudo-Cauchy sequence in $G$, i.e., that there is a pseudo-Cauchy sequence  $(a'_i)_{i\in I}$ in $G$ such that $a_i \equiv_m a'_i$ for every $i \in I$. We may assume that $\val^m(a_i-a_j)<\val^m(a_j-a_k)$ for all $i<j<k$. Let $I=\lambda$ for some limit ordinal $\lambda$. By transfinite induction, we can define the following sequence in $G$:
\begin{itemize}
\item For $\alpha=0$, let $a_0'$  be any element of $G$ such that $a_0'\equiv_m a_0$ and $\val(a_0')=\val^m(a_0)$.
\item For any $\alpha<\lambda$, let $b_{\alpha+1}$ denote an element in $G$ such that $b_{\alpha+1} \equiv_m a_{\alpha+1}-a_{\alpha}$ and $\val(b_{\alpha+1})=\val^m(a_{\alpha+1}-a_{\alpha})$. We set $a_{\alpha+1}'=a_\alpha'+ b_{\alpha+1}$.
\item Let $0<\alpha<\lambda$ be a limit ordinal, and assume $\val(a'_\beta-a'_\gamma)=\val^m(a_\beta-a_\gamma)$ with $a'_\beta \equiv_m a_\beta$, for eventually all $\beta<\gamma<\alpha$.
Then, $(a_{\beta}')_{\beta<\alpha}$ is pseudo-Cauchy in $G$ and, in particular, it admits a pseudo-limit $c_{\alpha}$ in $G$. Let $b_\alpha$ be such that $b_\alpha \equiv_m a_{\alpha}-c_{\alpha}$ and $\val(b_\alpha)=\val^m(a_{\alpha}-c_{\alpha})$. We set $a_\alpha'= c_{\alpha}+b_{\alpha}$.
\end{itemize}
For any $\alpha<\lambda$ and as long as the induction can run, we have $a_\alpha \equiv_m a'_\alpha$. To show that the induction runs for all $\alpha<\lambda$, we need to prove the following claim:
\begin{claim}
Let $\alpha$ be a limit ordinal, and suppose $(a_{\beta}')_{\beta<\alpha}$ is pseudo-Cauchy in $G$. Assume also $\val(a'_\beta-a'_\gamma)=\val^m(a_\beta-a_\gamma)$ with $a'_\beta \equiv_m a_\beta$, for eventually all $\beta<\gamma<\alpha$. Then,  $\val(a_\alpha'-a_\beta')=\val^m(a_\alpha-a_\beta)=\val^m(a_{\beta+1}-a_\beta)=\val(a_{\beta+1}'-a_\beta')$ eventually for all $\beta<\alpha$. In particular,  $a'_{\alpha}$ is a pseudo-limit of $(a'_{\beta})_{\beta<\alpha}$.
\end{claim}
\begin{proof}[Proof of the claim]
As $c_\alpha$ is a pseudo-limit of $(a'_\beta)_{\beta<\alpha}$, and $a_\alpha'=c_\alpha+b_\alpha$, it is sufficient to show that $\val(b_\alpha)>\val(c_\alpha-a'_\beta)=\val(a'_{\beta+1}-a'_\beta)$ eventually for all $\beta<\alpha$. Indeed, then we have 
\[\val(a_\alpha'-a_\beta')=\val(c_\alpha-a_\beta')=\val(a_{\beta+1}'-a_\beta')=\val^m(a_{\beta+1}-a_\beta)=\val^m(a_{\alpha}-a_\beta).\]
The third equality is by hypothesis and the last equality uses that the sequence $(a_{\beta} \mod mG)_{\beta<\alpha}$ is pseudo-Cauchy in $G/mG$.

Let $\beta<\alpha$ be sufficiently large. We have that 
\begin{align*}\val(b_\alpha)=\val^m(a_\alpha - c_\alpha) & =\val^m(a_\alpha - a_{\beta+1} + a_{\beta+1}- c_\alpha) \\
& \ge \min \{\val^m(a_\alpha-a_{\beta+1}),\val^m(a_{\beta+1}-c_\alpha)\}.
\end{align*}
On the other hand,
\begin{enumerate}
\item $\val^m(a_\alpha - a_{\beta+1})>\val^m(a_{\beta+1}-a_\beta)=\val(a_{\beta+1}'-a_\beta')$ by hypothesis, and
\item $\val^m(a_{\beta+1}- c_\alpha)=\val^m(a'_{\beta+1}-c_\alpha)\ge
\val(a_{\beta+1}'- c_\alpha)>\val(a_{\beta+1}'-a_\beta')$, the first equality following from $a_{\beta+1}'\equiv_m a_{\beta+1}$.
\end{enumerate}
Thus $\val(b_\alpha)>\val(a'_{\beta+1}-a'_{\beta})$, proving the claim.
\end{proof}
In particular, $(a'_\beta)_{\beta<\lambda}$ is pseudo-Cauchy, and it admits a pseudo-limit $a'$ in $G$. Then, $a' \mod mG$ is a pseudo-limit of the sequence $(a_\alpha \mod mG)_{\alpha < \lambda}$ in $G/mG$. Indeed, for any $\beta<\alpha<\gamma<\lambda$ large enough, we have 
\[\val^m(a'-a_\alpha)=\val^m(a'-a'_\alpha) \geq \val(a'-a'_\alpha)= \val(a'_\gamma-a'_\alpha)=\val^m(a_\gamma-a_\alpha)>\val^m(a_\alpha-a_\beta).\]
Thus $\val^m(a'-a_\beta)=\val^m(a_\alpha-a_\beta)$ for eventually all $\beta<\lambda$, and the statement is proved.
\end{proof}

\section{Towards a characterisation of stable embeddedness}\label{S:UniformValuation}

We begin by observing a fundamental fact that will be used throughout the entire section.

\begin{fact}
\label{FactDefinabilityConvexSubgroups}
Let $G$ be any ordered abelian group which is stably embedded. Then all cuts are definable in $G$. In particular all convex subgroups of $G$ are definable.
\end{fact}
\begin{proof}
If $G$ is stably embedded, clearly every cut in $G$ is definable. So if $C$ is a convex subgroup of $G$, in particular the cut $C^+$ is definable. Then, $C=\Inv(C^+)$ is definable as well.
\end{proof}

 We will show that, at least for a certain large subclass of ordered abelian groups, this is also a sufficient condition. We give a brief heuristic: let $G$ be an ordered abelian group, and assume that the archimedean valuation is interpretable in $G$. 
We will see that stable embeddedness can fail for exactly three (non exclusive) reasons:  
\begin{enumerate}
    \item the archimedian spine $\Gamma^a$ is not stably embedded (as a coloured chain),
    \item there exists $\gamma \in \Gamma_G^a$ such that the rib $R_\gamma^a$ of $G$ is not stably embedded as an ordered abelian group, or
    \item $(G,\val^{a})$ is not maximal.
\end{enumerate}

%Since there is a one-to-one correspondence between cuts of $\Gamma^a$ (which is a pure stably embedded coloured chain) and convex subgroups of $G$, Condition $(1)$ says that some convex subgroup is not definable.
We will see that obstructions of stable embeddedness in regular ordered abelian groups and in chains take the form of non-definable cuts. As a consequence, each of these three conditions relates to the existence of a certain type of non-definable cut. More precisely:
\begin{enumerate}
    \item if $(P,Q)$ is a non-definable cut in $\Gamma^a,$ then\\ $(val^{-1}(Q) \cup G_{<  0},val^{-1}(P) \cap G_{\geq 0})$ is a non definable cut in $G$,
    \item if $(P_\gamma,Q_\gamma)$ is a non-definable cut in a rib $R^a_\gamma$, $\{ a : a\mod V_\gamma \in P_\gamma\}$ and $\{ a : a\mod V_\gamma \in Q_\gamma\}$ induces a non-definable cut in $G$,
    \item if $(a_i)$ is a pseudo-Cauchy sequence in $G$ without pseudo-limit in $G$, any pseudo-limit $a$ in an elementary extension induces a non-definable cut $(G_{<a} , G_{>a} )$ in $G$.
\end{enumerate}
 The third should be treated with care (see Lemma \ref{PropositionStablyEmbeddedImpliesMaximal}).
As a corollary, we will have that such an ordered abelian group $G$ is stably embedded if and only if all cuts in $G$ are definable (see Corollary~\ref{CorMainTheorem2}).

\subsection{Ordered abelian groups with finite regular rank}\label{S:finite-rk}

We need first to analyse the case of regular ordered abelian groups. It is also worth to consider the slightly more general case of  ordered abelian groups with finite regular rank: one can obtain in this context 
 -- and without particular efforts-- a complete characterisation of uniform and non-uniform stably embedded models. This completes a previous work in \cite[Section 4]{CHY26}, that we will use as a reference. 
\subsubsection{Regular ordered abelian groups}

By Fact~\ref{F:RegChar}(5), we deduce that a regular ordered abelian group which is stably embedded is necessarily archimedean.

\begin{proposition}\label{PropArchimedeangroups}
Let $G\preccurlyeq H$ be an elementary pair of (non-trivial) regular ordered abelian groups. Then $G$ is stably embedded in $H$ if and only if all cuts $(L,R)$ of $G$ realised  in $H$ are definable cuts, i.e. are $\pm\infty$ or of the form $(\frac{g}{n})^\pm$ for $g\in G$, $n\in \mathbb{N}_{>0}$.\\
In particular, the following are equivalent:
\begin{itemize}
    \item $G$ is stably embedded,
    \item all cuts in $G$ are definable,
    \item $G$ is archimedean and $\div(G) \simeq \mathbb{R}$ if $G$ is dense, or $G\simeq \mathbb{Z}$.
\end{itemize}  
\end{proposition}
\begin{proof}
    By Fact~\ref{F:RegChar}(3), regular ordered abelian groups eliminate quantifiers in $\mathcal{L}_{\text{Pres}}$. It follows that we can reduce the definability of all $n$-types for all $n$ (over $G$) to the definability of all $1$-types: whenever $G\preccurlyeq H$ is an elementary extension of regular ordered abelian groups, and $\bar{a}=(a_0,\dots,a_{n-1})\in H^n$, we have
    
    \begin{equation}
\label{MarkerSteinhornforPresburger}
\underset{z_0,\dots,z_{n-1} \in \Z}{\bigcup}\text{tp}(\underset{i<n}{\sum}z_ia_i / G) \vdash \text{tp}(\bar{a}/G).
\end{equation}

Thus $tp(\bar{a}/G)$ is  definable if and only if $tp(\sum z_ia_i/G)$ is  definable for all $\bar{z}\in\mathbb{Z}^n$, and all types over $G$ realised in $H$ are (uniformly) definable if and only if all 1-types over $G$ realised in $H$ are (uniformly) definable.

Consider $p(x)=\tp(a/G)$  a non realised 1-type over $G$, where $a \in H$. By quantifier elimination, $p(x)$ is determined by the classes modulo $m$ of $a$ for all $m\in\mathbb{N}$, and by the cut $C^p=(L^p,R^p)$, where $L^p = \{d \in G \mid d<a\}$. (Note that $C^p$ determines the cut of $ma$ over $G$ for any $m$.)  The class modulo $m$ of $a$ is either not represented in $G$ or it is the class modulo $m$ of an element of $G$ --- it is in particular always definable. Thus, $p$ is definable if and only if the cut $C^p$ is definable in $G$. By \cite[Lemma 4.3.3]{CHY26} and quantifier elimination in divisible ordered abelian groups, the definable cuts in $G$ are precisely the cuts of the form $\pm\infty$ or of the form $(\frac{g}{n})^\pm$ for $g\in G$, $n\in \mathbb{N}_{>0}$. This shows the first part.

It also follows that $G$ is stably embedded if and only if all cuts in $G$ are definable, so are all of the form  $\pm\infty$ or $(\frac{g}{n})^\pm$. If $G$ is discrete, this happens if and only if $G \simeq \mathbb{Z}$, as $\mathbb{Z}$ is the only archimedean model of its theory. If $G$ is dense, this happens if and only if $G$ is archimedean and all cuts are of the form $\pm\infty$ or of the form $(\frac{g}{n})^\pm$, i.e., $\div(G)\simeq \mathbb{R}$.
\end{proof}

A consequence of the proof is uniform stable embeddedness of $(\mathbb{R},+,<)$ and of $(\mathbb{Z},+,<)$. More precisely, we get:

\begin{corollary}\label{Cor:StEmb}
\begin{enumerate}
    \item Up to isomorphism, $(\mathbb{R},+,<)$ is the unique stably embedded model of DOAG. Moreover, it is uniformly stably embedded. %(This follows, e.g., from the Theorem of Marker-Steinhorn,  Theorem~\ref{TheoremMarkerSteinhorn}.)
    \item Up to isomorphism, $(\mathbb{Z},+,<)$ is the unique stably embedded model of PRES. Moreover, it is uniformly stably embedded. %(See, e.g., \cite{CV}.)
\item Let $G$ be a dense non-divisible regular ordered abelian group. Then $T=\Th(G)$ admits a stably embedded model, but no uniformly stably embedded model (see \cite[Proof of Proposition~4.3.5]{CHY26}). 

Actually, for any such $T$ there are more than one stably embedded models up to isomorphism. There is a non-empty countable set $I$ and a family $(p_i)_{i\in I}$ of prime numbers such that, as groups, $G_1:=\bigoplus_i\mathbb{Z}_{(p_i)}\oplus\mathbb{R}\equiv (G,+)\equiv \bigoplus_i\mathbb{Z}_{p_i}$, as one may choose $(p_i)_{i\in I}$ such that $(G_1:pG_1)=(G_2:pG_2)=(G:pG)$ for every prime $p$.

The groups $G_1$ and $G_2$ are both of cardinality $2^{\aleph_0}$ and non-isomorphic (e.g., as $G_2$ contains no subgroup isomorphic to $\mathbb{Q}$).

It now suffices to choose group embeddings $\iota_i:G_i\hookrightarrow \mathbb{R}$ such that $\div(\iota_i(G_i))=\mathbb{R}$ for $i=1,2$. The ordered abelian groups one obtains on $G_1$ and $G_2$ using the orderings induced from  $\iota_1$ and $\iota_2$ are then non-isomorphic, and they are both stably embedded by Proposition~\ref{PropArchimedeangroups}.
\end{enumerate}
\end{corollary}

%Note that uniform stable embeddedness does not hold for any regular dense ordered abelian group (see the proof of \cite[Proposition~4.3.5]{CHY26} ). 

\subsubsection{Ordered abelian groups with finite regular rank}
 Let $G$ be an ordered abelian group  with \emph{finite regular rank}, i.e., with finitely many definable convex subgroups
\begin{equation*}
\Delta_0=\{0\} < \dots <\Delta_i < \dots < \Delta_{n}=G
\end{equation*}
In what follows, we will repeatedly use Fact~\ref{Fact:Purity-convex}.

For any $i<n$, we can associate to $G$ the following pure
%\footnote{Pure means here that elements of $\Delta_{i}$ divisible by an integer $n$ in $\Delta_{i+1}$ are also divisible by $n$ in $\Delta_i$. This is automatically true since $Q_i$ is torsion free.} 
short exact sequence of ordered abelian groups
\begin{equation}\label{EquationShortExactSequence}
0 \longrightarrow \Delta_i \overset{\iota}{\longrightarrow} \Delta_{i+1} \overset{\nu}{\longrightarrow} Q_i:=\Delta_{i+1}/\Delta_i \longrightarrow 0
\end{equation}
Observe that $Q_i$ can be divisible only if $i=0$. Indeed, if $Q_i$ were divisible for some $i>0$, $\Delta_{i+1}/\Delta_{i-1}$ would be regular by 
Fact~\ref{F:RegChar}(6), and the existence of the definable proper convex subgroup $\Delta_i/\Delta_{i-1}$ would thus contradict Fact~\ref{F:RegChar}(5).  We have the following characterisation for pure short exact sequences of abelian groups:
\begin{proposition}[\cite{Tou20b}]
\label{corPierre}
Let $\mathcal{M}=(A,B,C,\iota,\nu, \dots)$ be a pure short exact sequence of abelian groups, possibly with additional structure on the sort $A$ and on the sort $C$. Let $\mathcal{N}=(A',B',C',\iota',\nu',\ldots)$ be an elementary extension of $\mathcal{M}$. Then, $\mathcal{M}$ is (uniformly) stably embedded in $\mathcal{N}$ if and only if $A$ is (uniformly) stably embedded in $A'$ and $C$ is (uniformly) stably embedded in $C'$. 
\end{proposition}

This proposition applies in particular to short exact sequences of ordered abelian groups. We deduce the following:

\begin{theorem}
Let $G$ be an ordered abelian group with finite regular rank, and let $\Delta_0=\{0\} < \dots <\Delta_i < \dots < \Delta_{n}=G$ be all the definable convex subgroups of $G$. Set $Q_i:=\Delta_{i+1}/\Delta_i$, for $i<n$. Then, the  following are equivalent:
\begin{enumerate}
    \item $G$ is stably embedded (resp. uniformly stably embedded).
    \item $Q_i$ is stably embedded (resp. uniformly stably embedded) for every $i<n$.
    \item For every $i<n$, $Q_i\cong\Z$ or $Q_i$ is dense and $\div(Q_i)\cong\mathbb{R}$ (resp. $Q_i\cong\Z$ or $Q_i\cong\mathbb{R}$). 
\end{enumerate}
\end{theorem}

\begin{proof}
(1)$\Leftrightarrow$(2) follows from Proposition~\ref{corPierre}, by induction on $n$.
(2)$\Leftrightarrow$(3) follows from the  previous characterisation of (uniformly) stably embedded regular ordered abelian groups (Proposition~\ref{PropArchimedeangroups} and Corollary~\ref{Cor:StEmb}), since the $Q_i$'s are regular ordered abelian groups for all $i < n$ (by Fact~\ref{F:RegChar}).
\end{proof}

\begin{corollary}\label{CorollaryZnUniformlySE}
The ordered abelian groups $\Z^n$ and $\Z^n \times \mathbb{R}$  are uniformly stably embedded, for every $n \in \mathbb{N}$, and they are the unique models of their own theory which are stably embedded. 
\end{corollary}

\begin{remark}
In fact, $\Z^n$ and $\Z^n \times \mathbb{R}$ are the only ordered abelian groups of finite regular rank which are uniformly stably embedded. Any other complete theory of ordered abelian groups of finite regular rank does admit  stably embedded models, but no uniformly stably embedded model. (This may be easily derived from our analysis. The details are left to the reader.) 
\end{remark}

\subsection{Ordered abelian groups with interpretable regular valuation} \label{SS:RegularValuation}
In this subsection, we focus on a large class of ordered abelian groups, interpreting a compatible valuation that we call the \emph{regular valuation}. This assumption allows us to use the tools of valuation theory described in the preliminary section; our characterisation of stably embedded pairs will then take the form of a transfer principle for valued abelian groups.

Recall that, as in Definition~\ref{DefinitionAuxiliarySorts}, we use the convention that the union over the empty set is $\emptyset$ and the intersection over the empty set is $G$. We cite first a result of Delon and Farré, adapted to our notation\footnote{
One needs to see that for $g\in G\setminus \{0\}$ the group denoted by $\mathcal{A}_N(g)$ in \cite{Sch82} is equal to $G_{\mathfrak{t}_N(g)}$.  This is observed in \cite[Section 1.5]{CH11} and follows from the fact that the quotient $G_{\mathfrak{t}_N^+(g)}/G_{\mathfrak{t}_N(g)}$ is $N$-regular.}: 

\begin{fact}[{\cite[Theorem 4.1]{DF96}}]\label{FactDelonFarre}
Let $G$ be an ordered abelian group. If $H$ is a definable convex subgroup, then there is an integer $N$ such that:
\[H := \bigcap_{g\notin H } G_{\mathfrak{t}_N(g)}
 \]
\end{fact}

We wish to study the following property (U), which states that all fundaments of $G$ are uniformly definable, in a strong sense:
\begin{enumerate}
    \item[(U)] The archimedian valuation is interpretable in the language of ordered groups: there is a formula $\phi(x,y)$ with $\vert y \vert =1$ such that for all $g\in G$, $\phi(G,g)=V_g^a$.
\end{enumerate}

Notice first that this property is not preserved under elementary equivalence. Indeed, we can observe the following:
\begin{remark}
     \begin{enumerate}
         \item Let $G$ be an ordered abelian group and $H=\phi(G)\leq G$ a proper convex definable subgroup, such that $G/H$ is dense. Then there is $G'\succcurlyeq G$ such that, setting $H'=\phi(G')$, the quotient group $G'/H'$ has no smallest non-trivial convex subgroup. 
         \item Let $G$ be an ordered abelian  group and $g\in G$. Assume that $V_g^a$ is definable given by a formula $\phi(x,g)$ and $G/V_g^a$ is dense. Then there is an elementary extension $G'$ of $G$ such that $\phi(G',g)$ is not an archimedean fundament.
         \end{enumerate}
\end{remark}

\begin{proof}
Assume that $H=\phi(G)$ is such that $G/H$ is a non-trivial dense ordered abelian group. We construct by induction an increasing sequence of elementary extensions $(G_n)_n$, where for $n\geq 0$, $G_{n+1}/\phi(G_{n+1})$ contains a realisation of $0^+$ in $G_{n}$. Then the union $G':= \bigcup_n G_n$ is an elementary extension of $G=G_0$ such that $G'/\phi(G'$) has no smallest non-trivial convex subgroup. This proves (1).

(2) follows from (1), setting $H:=V_g^a$.
\end{proof}

We want to characterise ordered abelian groups which are elementarily equivalent to a group satisfying (U). For that, we introduce the notion of regular spine, which is the definable counterpart of the archimedean spine.
\begin{definition}[Regular valuation]
    Let $G$ be an ordered abelian group, $g\in G$. We denote by
    \begin{itemize}
        \item $V^r_g$ the union of all \emph{definable} convex subgroups not containing $g$ if $g\neq 0$, and $\emptyset$ otherwise, called the \emph{regular fundament of $g$}; %(the largest $\bigvee$-definable convex subgroup not containing $g$
        \item $C^r_g$ the intersection of all \emph{definable} convex subgroup containing $g$, called the \emph{regular cover of $g$}; %the smallest $\bigwedge$-definable convex subgroup containing $g$.
        \item (assuming $g\neq 0$) $R^r_g$ the group quotient $C^r_g/V^r_g$ called the \emph{regular class} or the \emph{regular rib} at $g$; 
        \item $\Gamma^r$ the quotient $G/\sim^r$ called the \emph{regular spine} (where the equivalence relation $\sim^r$ is given by: $g \sim^r {g'}$ if and only if $V^r_g=V^r_{g'}$, equivalently $C^r_g=C^r_{g'}$; it is equipped with the reverse order with respect to the inclusion);
        \item $\val^r:G \rightarrow \Gamma^r$ the canonical projection map, called the \emph{regular valuation}.
    \end{itemize}
\end{definition}

\begin{lemma}
If $G$ is an ordered abelian group, $(G,val^r)$ is a $\Z$-invariant valued abelian group. Moreover, for any $g\in G$, the regular rib $R^r_g$ is a regular ordered abelian group.    
\end{lemma}

\begin{proof}
It is clear that  $(G,val^r)$ is a $\Z$-invariant valued abelian group. Moreover, for every $n>1$ and $g\in G$, the regular rib $R^r_g$ is a subquotient of the $n$-regular ordered abelian group $G_{\mathfrak{t}_n^+(g)}/G_{\mathfrak{t}_n(g)}$. Thus $R^r_g$ is $n$-regular for every $n \in \mathbb{N}_{>1}$, so regular by Proposition~\ref{F:RegChar}. 
\end{proof}

\begin{definition}\label{{DefinitionRegularValuation}}
    We denote the following property for ordered abelian groups:
    \begin{itemize}
        \item[(UR)$_{\phantom{N}}$] The regular valuation is interpretable in the language of ordered groups: there is a formula $\phi(x,y)$ where $\vert y \vert =1$ such that for all $g\in G$, $\phi(G,g)=V^r_g$. 
    \end{itemize}

    This property will be assumed for the rest of the paper. In this paragraph, we will also need to discuss the following properties: for $N\in \mathbb{N}_{>1}$
    \begin{itemize}
        \item[(UR)$_N$] For all $g\in G$, $V^r_g=G_{\mathfrak{t}_N(g)}$.
    \end{itemize}
    For the notation involved, see Definition~\ref{DefinitionLsyn}.

\end{definition}   

\begin{remark}
    (UR)$_N$ is an elementary property, for it is equivalent to the following:
    \begin{itemize}
        \item[(UR')$_N$] for all $g\in G$, $G_{\mathfrak{t}_N(g)}$ is maximal within $\{G_\alpha \ \mid \ \alpha \in \bigcup_n \mathcal{T}_n \}$ with the property of not containing $g$.
    \end{itemize}

\end{remark}
Notice that if $G$ satisfies (UR)$_N$ for some $N\in \mathbb{N}$, we can identify $\Gamma^r$ with $\mathcal{T}_N$ and $\val^r$ with $\mathfrak{t}_N$.

\begin{proof} 
    This is almost a reformulation of (UR)$_N$, the non-trivial implication being (UR')$_N \Rightarrow $(UR)$_N$.  Let $g\in G$ and assume  $G_{\mathfrak{t}_N(g)}$ is maximal  within $\{G_\alpha \ \mid \ \alpha \in \bigcup_n \mathcal{T}_n \}$ with the property of not containing $g$. By Fact~\ref{FactDelonFarre}, any definable convex subgroup which doesn't contain $g$ is an intersection of convex subgroups in $\{G_\alpha \ \mid \ \alpha \in \bigcup_n \mathcal{T}_n \}$ which do not contain $g$. It follows therefore that $G_{\mathfrak{t}_N(g)}$ is the largest definable convex subgroup which does not contain $g$. So $G_{\mathfrak{t}_N(g)}=V^r_g$, as wanted.
\end{proof}

If a group $G$ satisfies (U), clearly the regular valuation coincides with the archimedean valuation and it will also satisfy (UR). The next proposition shows that (UR) is preserved under elementary equivalence. In particular if an ordered abelian group $G$ is elementarily equivalent to another ordered group satisfying (U), then it satisfies (UR).

\begin{proposition}\label{PropositionUR}
    Let $G$ be an ordered abelian group. Then the following are equivalent:
    \begin{enumerate}
        \item $G$ satisfies (UR),
        \item there is $N\in \mathbb{N}$ such that $G$ satisfies (UR)$_N$.
    \end{enumerate}
\end{proposition}

We derive a proof of this proposition from Fact~\ref{FactDelonFarre}:

\begin{proof}

The proof of $(2) \Rightarrow (1)$ is immediate by interpretability of $\mathfrak{t}_N$.

We show now $(1) \Rightarrow (2)$. Assume $(1)$, witnessed by a formula $\phi(x,y)$. For all $N\in \mathbb{N}_{>1}$ and $g\in G$, $\phi(G,g) \supseteq G_{\mathfrak{t}_N(g)}$ as $\phi(G,g)$ is the largest definable convex subgroup not containing $g$. This inclusion still holds in an $\omega$-saturated elementary extension, and so we may assume that $G$ is $\omega$-saturated. Let us prove $\phi(G,g) \subseteq G_{\mathfrak{t}_N(g)}$ for some $N$.  By Fact~\ref{FactDelonFarre} and compactness, there are integers $N_0,\dots, N_{k-1}$ such that for all $g\in G$, $\phi(G,g)$ is an intersection of $G_\alpha$ for $\alpha\in  \mathcal{T}_{N_i}$, $i<k$ which do not contain $g$. Let $N$ be the product of all $N_i$. Since $G_{\mathfrak{t}_N(g)}$ is the largest $G_\alpha$ with the property of not containing $g$ for $\alpha \in \mathcal{T}_N$ and $G_{\mathfrak{t}_N(g)} \supseteq G_{\mathfrak{t}_{N_i}(g)}$ for all $i<k$, we have $\phi(G,g) \subseteq G_{\mathfrak{t}_N(g)}$. 
\end{proof}

\begin{Notation}
The property (UR) will be assumed for most of the remainder of the paper. In order to simplify the notation, if $G$ satisfies (UR), we will drop the exponent $r$ and write 
    $V_g, C_g, \Gamma$ and $\val$ instead of $V_g^r, C_g^r, \Gamma^r$ and $\val^r$, respectively. 
\end{Notation}
%\todo[inline]{Discuss this fact:
%Note that by Proposition \ref{PropositionUR}, if $G$ satisfies (UR) , then the regular valuation $\val:G \rightarrow \Gamma$  is the finest valuation that we can interprets on $G$ in the language of ordered abelian groups.}

In the same way that an ordered abelian group $G$ is regular if and only if it is elementarily equivalent to an archimedean one, we will see, under some additional hypothesis, that $G$ satisfies (UR) if and only if it is elementarily equivalent to an ordered abelian group satisfying (U).

\begin{remark}\label{RemarkIdentificationSmVm}
    Assume $G$ satisfies (UR) and let $\val$ be the regular valuation. Then the induced $m$-valuation (as defined for arbitrary valued groups in Definition~\ref{DefinitionmValuation})  $\val^m$ can be identified with $\mathfrak{s}_m$. Indeed for $a\in G$, $V^m_a:=\{x\in G \ \vert \ \forall \ g\in G, \ \val(x)>\val(a+mg)\} 
    = \bigcap_{c\in a+mG}V_c$ is a convex subgroup not intersecting $a+mG$, and therefore it is a subgroup of $G_{\mathfrak{s}_m(a)}$.  If $c\in a+mG$, then $c\notin G_{\mathfrak{s}_m(a)}$ and $G_{\mathfrak{s}_m(a)} \subseteq V_c$ (since $G_{\mathfrak{s}_m(a)}$ is an intersection of $V_b$'s not containing $c$, and $V_c$ is the largest $V_b$ not containing $c$ ).  We get this other inclusion: $G_{\mathfrak{s}_m(a)}\subseteq V^m_a$. 
\end{remark}

\subsection{A quantifier elimination result for ordered abelian groups with interpretable regular valuation.}\label{SS:QuantifierEliminationUniformValuation}

Let $G$ be an ordered abelian group satisfying (UR).  For the rest of the paper, we will restrict our study to a smaller class, by assuming furthermore that $(G,\val)$ satifies the property (M). Recall that (M) consists of the following scheme of axioms:
\begin{align}
    \tag{M} (\forall \ x\in G)(\exists \ y\in G) \ x\equiv_n y \wedge \val^n(x)=\val(y), \text{ for all $n\in\mathbb{N}_{>0}$}.
\end{align}

Notice that these axioms are expressible in the language of ordered abelian groups since the regular valuation is interpretable in this language. This extra assumption will allow us to explicit a simple and useful multisorted language where these ordered abelian groups eliminate group-sorted quantifiers. We already observed in Lemma~\ref{LemmaValueSetModm} that (M) holds in particular if $(G,\val)$ is maximal.

\begin{remark}
    The property (UR) and  maximality with respect to the archimedean valuation are unrelated. For instance, the ordered abelian group $\hprod_{p \in \mathbb{P}} \mathbb{Z}_{(p)}$ is a maximal ordered abelian group with all principal convex subgroups definable, but it does not interpret the archimedean valuation (by Fact~\ref{FactDelonFarre} and compactness). Conversely, the ordered abelian group generated by $\sum_{i\in \mathbb{N}}\mathbb{Z}$ and the element $(2,2,\dots)$ satisfies (U) but not (M). %See Appendix \ref{AppendixElementaryClassOrderedAbelianGroups}.
\end{remark}

We will see that if $G$ satisfies (UR) and $(G,\val)$ satisfies (M) then it is elementarily equivalent to a maximal ordered abelian group satisfying (U). We define first the language with which we will work for the rest of the paper.

\begin{definition}\label{DefinitionRegularValuation}%[Case of uniformly definability of principal convex subgroups] 
Assume $G$ satisfies (UR) and assume $(G,\val)$ satisfies (M).
 Let $\mathcal{L}$ be the language consisting of 
\begin{itemize}        
\item the main sort $G$, with the symbols $+,-,0,<$ interpreted in the obvious way,
\item an auxiliary sort $\Gamma$ for $\Gamma_G$, with a binary relation $<$, interpreted by the ordering relation on $\Gamma_G$; a unary predicate $C_\phi(x)$ on $\Gamma$ for each sentence $\phi$ in $\mathcal{L}_{\text{oag}}$, defined by, for any $\gamma \in \Gamma$, $C_\phi(\gamma)$ if and only if  $R_\gamma \models \phi$, 
\item a function symbol $\val^{m}$ for each non-negative integer $m$, interpreted by the map $\val^{m}:G \rightarrow \Gamma^m$ as defined in (\ref{inducedvaluation}),
\item a unary predicate $x=_\bullet k_\bullet$ on $G$ for each $k \in \Z \setminus \{0\}$, defined by, for any $a \in G$, $a =_\bullet k_\bullet$ if the quotient $G/V_{\val(a)}$ is discrete and $a \mod V_{\val(a)}$  is $k$ times the minimal positive element of $G/V_{\val(a)}$,
\item a unary predicate $x \equiv_{\bullet m} k_\bullet$ on $G$ for each $m \in \mathbb{N}, m>0$ and $k \in \{1,\dots,m-1\}$, defined by, for any $a \in G$, $a \equiv_{\bullet m} k_\bullet$ if $G/V_{\val^m(a)}$ is discrete and $a \mod V_{\val^m(a)}$ is congruent modulo $m$ to $k$ times the minimal positive element of  $G/V_{\val^m(a)}$.
\end{itemize}
\end{definition}
By Corollary~\ref{CorValueSetmodm}, for an integer $m>1$,  if $\phi_m$ denotes the  $\mathcal{L}_{\text{oag}}$-sentence $\neg\forall x \exists y\ my=x$, then the unary predicate $C_{\phi_m}$ defines the $m$-value set $\Gamma^m$.

\begin{theorem}\label{TheoremQuantifierEliminationOneValuation}
%Let $G$ be an ordered abelian group. Assume that $G$ satisfies (UR), let $\val$ be the regular valuation and assume $(G,\val)$ satisfies (M). 
For every $N \in \mathbb{N}_{>1}$, the theory $T_N$ of ordered abelian groups satisfying (UR)$_N$ and (M) eliminates quantifiers relatively to the sort $\Gamma$ in the language $\mathcal{L}$. 
\end{theorem}

 In particular, maximal ordered abelian groups with interpretable archimedean spine eliminate quantifiers in the language $\mathcal{L}$.   
\begin{proof}
We deduce the statement from the more general quantifier elimination result due to Cluckers and Halupczok (Fact~\ref{FactQUantifierEliminationLsyn}). For the notation involved, see Definition~\ref{DefinitionLsyn}. We have to interpret every symbol in $\mathcal{L}_{syn}$ using formulas in the language $\mathcal{L}$ with no group quantifiers. 
 We note first that, for $m>1$, the $\mathcal{L}_{syn}$-formula $x\equiv_m 0$ is equivalent to the quantifier-free formula $\val^m(x)=\infty$. By Remark~\ref{RemarkIdentificationSmVm}, for every $m\in \mathbb{N}_{>1}$, $\mathcal{S}_m$ can be identified with $\Gamma^m$, and the $m$-valuation map $\val^m$ corresponds to the map $\mathfrak{s}_m$. Now we show that we can interpret the sorts $\mathcal{T}_m$, $\mathcal{T}_m^+$ and the projection $\mathfrak{t}_m$ for all $m$ without using quantifiers over the group sort. Take any $G\models T_N$, let $a\in G$ and let $\beta=\mathfrak{t}_m(a)$. The convex subgroup $G_\beta$ is the union of all the convex subgroups $G_\alpha$ which  do not contain $a$ and where $\alpha\in \Gamma^m=\mathcal{S}_m$ i.e. 
\[G_\beta= \{ x\in G \ \vert \ \exists \delta'\in \Gamma^m \ \val(x) \geq \delta' > \val(a)\}.\]

Thus, $\mathcal{T}_m$ can be interpreted in $\Gamma$ as the quotient $\Gamma/ \sim_m$ where the equivalence relation is given by:
\[\gamma \sim_m \gamma' \quad \text{ if and only if }\quad  \forall \delta \in \Gamma \ (\exists \delta'\in \Gamma^m \ \delta \geq \delta' > \gamma) \leftrightarrow \ (\exists \delta'\in \Gamma^m \ \delta \geq \delta' > \gamma'),\]
%We interprets naturally the ordering $<$ between %$\mathcal{T}_m$ and $\mathcal{T}_k$ by 
%\[[\gamma]_{\sim_m}\leq [\gamma']_{\sim_k} \quad \text{ if and only if }\quad \{\delta \in \Gamma \ \vert \ \exists \delta'\in \Gamma^m \ \delta \geq \delta' > \gamma\}\supseteq \{\delta \in \Gamma \ \vert \ \exists \delta'\in \Gamma^k \ \delta \geq \delta' > \gamma'\}. \]
and the map $\mathfrak{t}_m: G \rightarrow \mathcal{T}_m$ is interpreted by the quotient map $a\in G \mapsto [\val(a)]_{\sim_m}$.

Similarly, if $a\in G$ and $\beta^+=\mathfrak{t}_m^+(a)$, the convex subgroup $G_{\beta^+}$ is the intersection of all convex subgroups $G_\alpha$ which contain $a$, where $\alpha\in \Gamma^m=\mathcal{S}_m$ i.e.
 \[G_{\beta^+}=\{ x\in G \ \vert \ \forall \gamma \in \Gamma^m \ \gamma \leq \val(a) \rightarrow \gamma \leq \val(x)\}.\]
 It follows that $\mathcal{T}_m^+$ can be interpreted as $\mathcal{T}_m$ without quantifier in the group sort. The ordering between elements in $\bigcup_m\mathcal{T}_m \cup \mathcal{T}_m^+$ can also be interpreted without quantifiers in the group sort. For instance, for $m,k\in \mathbb{N}_{>1}$ and $a,b \in G$, we have $\mathfrak{t}_m(a)\leq \mathfrak{t}_k^+(b)$ if and only if
\[  \{\delta \in \Gamma \ \vert \ \exists \delta'\in S_m \ \delta \geq \delta' \geq \val(a)\}\supseteq  \{\delta \in \Gamma \ \vert \ \forall \delta'\in S_k \ \delta' \leq \val(b) \rightarrow \delta' \leq \delta\}. \]

 Finally, we show that the predicates $D_{p^r}^{[p^s]}(x)$ on the group sort are not required, by showing that for all $\alpha \in \mathcal{S}_m$, $G_{\alpha}^{[m]}=C_\alpha' + mG$, where $C_\alpha'$ is a convex subgroup containing $G_{\alpha}$. Recall that for $m\in \mathbb
 {N}_{>1}$ and $\alpha \in  \mathcal{S}_m$:

\begin{equation*}
G_{\alpha}^{[m]} = \bigcap_{\substack{H \supsetneq G_\alpha \\ H \text{ convex subgroup of } G}} (H + mG).
\end{equation*}
This intersection can be restricted to convex subgroups of the form $G_{\beta}$ for $\alpha>\beta \in \mathcal{S}_m$ (\cite[Lemma 2.4]{CH11}). Let $G_\beta$ such a convex subgroup, then $G_\beta+mG= \{x \ \vert \ \val^m(x) > \beta\}$. Indeed $G_\beta+mG \supseteq \{x \ \vert \ \val^m(x) > \beta\}$ follows from condition (M) and the other inclusion is clear. Then going back to the definition of $G_{\alpha}^{[m]}$, we have that 
\begin{align*}
    G_{\alpha}^{[m]} &= \bigcap_{\{ \beta\in \mathcal{S}_m\mid \beta <\alpha\}} G_\beta + mG\\
&   = \{x \ \vert \ \forall \beta \in \mathcal{S}_m\  (\beta< \alpha  \rightarrow \ \val^m(x)> \beta)\}\\
& 
= G_{\alpha'} +mG.
\end{align*}
where $\alpha'$ is the immediate predecessor of $\alpha$ in $\Gamma^m$ if it exists, or is equal to $\alpha$ otherwise. 
It follows that we can express the predicate $D^{[p^s]}_{p^r}(x)$ for $s>r$ in $\mathcal{L}$ without quantifying over the group sort. All unary predicate symbols on the sort $\mathcal{A}$ of $\mathcal{L}_{syn}$ (such as the predicate discr) correspond to predicates $C_\phi$ in $\mathcal{L}$ for some appropriate sentences $\phi$. To conclude, we only need to observe that since we have only quantified over $\Gamma$ in order to recover the language $\mathcal{L}_{syn}$, every group-sorted-quantifier-free formulas in the language $\mathcal{L}_{syn}$ is equivalent to a group-sorted-quantifier-free $\mathcal{L}$-formula. Therefore, by Fact \ref{FactQUantifierEliminationLsyn}, every formula $\phi(x)$ in $\mathcal{L}$ is equivalent to a group-sorted-quantifier-free $\mathcal{L}$-formula $\psi(x)$ and this concludes our proof. 
One can notice also that, we can choose $\psi$ independently on $N$, i.e., such that $T_N\vdash \forall x \  \phi(x)\leftrightarrow \psi(x)$ for every $N\in \mathbb{N}_{>1}$. 
\end{proof}

\begin{corollary}\label{Cor:Chain-Rib-StEmb}
Let $G$ be an ordered abelian group satisfying (UR), such that  $(G,\val)$ satisfies (M). Then the following hold:
\begin{enumerate}
    \item The auxiliary sort $\Gamma$ is stably embedded (as a sort of $G^{eq}$) with induced structure that of the pure coloured chain ${(\Gamma,(C_{\phi})_{\phi\in \mathcal{L}_{\text{oag}}},<)}$.
    \item For every $\gamma \in \Gamma$, the rib $R_\gamma$ is stably embedded (as a sort of $(G,\gamma)^{eq}$, where the parameter $\gamma$ is named as constant) with induced structure that of a pure regular ordered abelian group.
\end{enumerate}
\end{corollary}

\begin{proof}
Stable embeddedness and purity of $\Gamma$ follow directly from relative quantifier elimination. The ribs $R_\gamma$ do not have definable convex subgroups and thus must be regular. Stably embeddedness and purity of each rib $R_\gamma$ follow by applying Fact~\ref{Fact:Purity-convex} twice.
\end{proof}

Notice that a morphism of ordered abelian groups $i:G \hookrightarrow G'$ induces a morphism of chains $f_\Gamma^a:\Gamma^a \rightarrow {\Gamma^a}', \val(a)  \mapsto  \val'(i(a))$ between the archimedean spines $\Gamma^a$ and ${\Gamma^a}'$.
%\[\begin{array}{cccc}
    % f_\Gamma^a: &  {\Gamma^a} &\rightarrow & {\Gamma^a}'\\
     %& .
%\end{array}  \]
It is easy to see that this definition does not depend of the choice of a representative. However, in general, a morphism of ordered abelian groups does not induce a morphism between the regular spines (consider for instance the embedding $\mathbb{Q}\times \mathbb{Z}_p \hookrightarrow  \mathbb{Q}\times \mathbb{Z}_p\times \mathbb{Z}_p, (a,b)\mapsto (a,0,b)$).
This motivates the following definition:
%Old definition. Only preserve Sn
%\begin{definition}
%    An embedding of ordered groups $\iota: G \rightarrow G'$ is called \emph{regular} if for every integer $n$ and all elements $a,b\in G$, $[a,b]$ contains an element divisible by $n$ if and only if $[f(a),f(b)]$ contains an element divisible by $n$.
%\end{definition}
\begin{definition}
    An embedding of ordered groups $\iota: G \rightarrow G'$ is called \emph{regular} if for every integer $N$ and all elements $a,b\in G$, $\mathfrak{t}_N(a)\leq \mathfrak{t}_N(b)$ if and only if $\mathfrak{t}_N(\iota(a)) \leq \mathfrak{t}_N(\iota(b))$ . \end{definition}
Recall that $G$ is regular if and only if $\{0\}$ is the unique definable convex subgroup. Then, in particular, if $i: G \hookrightarrow G'$ is a regular embedding and $G'$ is regular, then $G$ is regular. More generally:
\begin{observation}
A regular embedding $i: G \hookrightarrow G'$ induces a morphism of chains between the regular spines:
\[\begin{array}{cccc}
     f_\Gamma=f_\Gamma^r: &  \Gamma &\rightarrow & \Gamma'\\
     & \val(a) & \mapsto & \val'(i(a)).
\end{array}  \]

\end{observation}
\begin{proof}
    By Fact \ref{FactDelonFarre}, every definable convex subgroup $\Delta$ is an intersection of convex subgroups of the form $G_{\mathfrak{t}_N}(g)$, for some $N\in \mathbb{N}$ and $g\in G$. Since $\iota$ is a regular embedding, for all $a,b\in G$ we get the following equivalences: 
\begin{itemize}
    \item[] There is a definable convex subgroup $\Delta \subseteq G$ such that $a \in \Delta$ and  $b\notin \Delta$,
    \item[$\Leftrightarrow$] There is $N \in \mathbb{N}$ such that  $\mathfrak{t}_N(a)>\mathfrak{t}_N(b)$,
    \item[$\Leftrightarrow$] There is $N \in \mathbb{N}$ such $\mathfrak{t}_N(\iota(a))>\mathfrak{t}_N(\iota(b))$,
    \item[$\Leftrightarrow$] There is a definable convex subgroup $\Delta' \subseteq G'$ such that $\iota(a) \in \Delta'$ and  $\iota(b)\notin \Delta'$.
\end{itemize}
It follows that $\iota$ preserves regular classes and induces therefore an embedding $\Gamma \rightarrow \Gamma' $ of the regular spines.
\end{proof}

%% Old proof below. There is a problem with the equivalence of separation. This is false.
%\begin{proof}
%    Notice that for $y,a \in G$ and $n\in \mathbb{N}_{>1}$, $\Braket{y}\cap a+nG \neq  \emptyset$ if and only if  $[a,a+ny] \cap nG \neq \emptyset$ and using the assumption of regularity, this happens if and only if $\Braket{f(y)}\cap f(a)+nG'\neq  \emptyset$. It follows that 
%    for $\alpha=\mathfrak{s}_n(a) \in \mathcal{S}^n(G)$, since $G_{\alpha}= \{y \ \vert \ \Braket{y}\cap a+nG =  \emptyset \}$  we have  $f(G_\alpha)=G'_{\alpha'}\cap f(G)$ where $\mathfrak{s}_n(f(a)):=\alpha'\in \mathcal{S}^n(G')$. In particular, for all pairs $a,b$, there is a definable convex subgroup separating $a$ and $b$ if and only if there is a definable convex subgroup separating $f(a)$ and $f(b)$. This means that $V_a=V_b$ in $G$ if and only if we have $V_{f(a)}=V_{f(b)}$ in $G'$. Similarly, $V^n_a=V^n_b$ if and only if $V^n_{f(a)}=V^n_{f(b)}$ for $n>1$. This concludes our proof. 
%\end{proof}

Another immediate corollary is the following:
\begin{corollary}\label{CorollaryAKEForOrderedAbelianGroups}
     Let $G$ and $G'$ be two ordered abelian groups satisfying $(UR)$ and $(M)$. Then
     \begin{enumerate}
        \item $G \equiv G'$ if and only if \[(\Gamma,(C_{\phi})_{\phi\in \mathcal{L}_{\text{oag}}},<) \equiv (\Gamma',(C_{\phi}')_{\phi\in \mathcal{L}_{\text{oag}}},<)\] 
         \item if $G\subseteq G'$, then $G \preccurlyeq G'$ if and only if the following three conditions hold:
         \begin{enumerate}
             \item The embedding $G\subseteq G'$ is regular.
         \item The induced embedding of the regular spines $f_\Gamma:\Gamma \rightarrow \Gamma'$ preserves the colors $C_\phi$, and $(\Gamma,(C_{\phi})_{\phi},<) \subseteq (\Gamma',(C_{\phi}')_{\phi},<)$ is elementary.
    \item For every $\gamma\in \Gamma$, the induced embedding of regular ribs $(R_\gamma,+,0,<) \subseteq (R_\gamma',+,0,<)$ is elementary.
    \end{enumerate} 
     \end{enumerate}

\end{corollary}
Notice that \cite[Corollary 4.7]{Sch82} is similar to the first point, and holds for arbitrary ordered abelian groups.  However, we are not aware of a similar criterion for elementary extensions.

\begin{proof}
    In both points, the left-to-right implication is immediate.
    The right-to-left implication in the first point follows from the following fact: a partial $\mathcal{L}$-isomorphism $G\rightarrow G'$ between such ordered abelian groups -- where $\mathcal{L}$ is as introduced in Definition \ref{DefinitionRegularValuation}-- is elementary if the induced partial isomorphism $\Gamma \rightarrow \Gamma'$ is elementary. One may simply consider the partial isomorphim $f:\{0\in G\} \rightarrow \{0\in G'\}$, which induces the partial isomorphism $\{\infty\} \rightarrow \{\infty\}$ between $(\Gamma,(C_{\phi})_{\phi\in \mathcal{L}_{\text{oag}}},<)$ and $(\Gamma',(C_{\phi}')_{\phi\in \mathcal{L}_{\text{oag}}},<)$ . Since the latter is elementary by assumption, $f$ is elementary and  $G\equiv G'$.

    To show the second point, let $G\subseteq G'$ be ordered abelian groups satisfying $(M)$ and (UR) such that 
         \begin{align}\label{EquationGammaGammaPrime}
         (\Gamma,(C_{\phi})_{\phi\in \mathcal{L}_{\text{oag}}},<) \preccurlyeq (\Gamma',(C_{\phi}')_{\phi\in \mathcal{L}_{\text{oag}}},<).
     \end{align}
     and for all $\gamma \in \Gamma$, 
      \begin{align}\label{EquationRgammaRgammaPrime}
     (R_\gamma,+,0) \preccurlyeq (R_\gamma',+,0).
     \end{align}
     Then we show first that 
    $G$ is a substructure of $G'$ in the language $\mathcal{L}$.
    We have obviously that $C_\phi^G=C_\phi^{G'}\cap \Gamma$. Let $a\in G$. Then for all $m\geq 0$, there are $g\in G, g'\in G'$
 such that $(\val^m)^{G
 }(a)=\val(a-mg)\eqqcolon\gamma$ and $(\val^m)^{G'
 }(a)=\val(a-mg')$. Notice this holds also for $m=0$ with the convention that $\val^0=\val$. If $(\val^m)^{G
 }(a)<(\val^m)^{G'
 }(a)$, this would mean that $a-mg \mod V_\gamma =mg'-mg \mod V_\gamma $ is divisible by $m$ in $R_\gamma'$ but not in $R_\gamma$, contradicting $R_\gamma \preccurlyeq R_\gamma'$. Thus $(\val^m)^{G}=(\val^m)^{G'}$.
 We deal similarly with the predicate $\equiv_{\bullet m} k$ (with the convention that $\equiv_{\bullet 0}$ is $=_\bullet$), using the fact that  $G\models a \equiv_{\bullet m}k$ if and only if $a-mg \equiv_m k \mod V_\gamma$, where $g$ and $\gamma$ are as previously. This shows that $G$ is a substructure of $G'$ in $\mathcal{L}$.

     We want to show now that $G\preccurlyeq G'$. By relative quantifier elimination, every formula $\phi(x,y)$ with tuples of variables $x,y$ in $G^{\vert x\vert} \times \Gamma^{\vert y\vert}$ is equivalent to a boolean combination of formulas of the form:
\begin{itemize}
		\item[a)] $P(x) > 0$,
		\item[b)] $P(x)\equiv_{\bullet m} k_\bullet$,
		\item[c)] $P(x) =_\bullet k_\bullet$,
		\item[d)] $\psi(\val^{m_0}(P_0(x)),\dots,\val^{m_{h-1}}(P_{h-1}(x)),y)$,
	\end{itemize}
where $\psi$ is a formula in $(\Gamma,(C_{\phi})_{\phi\in \mathcal{L}_{\text{OAG}}},<)$,  $P,P_0,\dots,P_{h-1} \in \Z[X] \setminus \{0\}$, $m_0,\dots,m_{h-1} \in \mathbb{N}$, $m,k \in \mathbb{N} \setminus \{0\}$ .
We may check separately for each of these formula $\chi(x,y)$, and for each tuples $(a,\theta)\in G^{\vert x\vert}\times \Gamma^{\vert y\vert}$, that $G \models \chi(a,\theta) $ if and only if $G' \models \chi(a,\theta)$. The case of the formulas a),b) and c) follows immediately from the fact that $G$ is a substructure of $G'$ in the language $\mathcal{L}$. We now deal with the formulas of the form d). Since $G$ is a substructure of $G'$, for every function symbols $\val^m$ we have $\restriction {(\val^{m})^{G'}}{G}=(\val^{m})^{G}$. By (\ref{EquationGammaGammaPrime}), it follows that  for all $(a,\theta)\in G^{\vert x\vert}\times \Gamma^{\vert y \vert}$, 
\[G' \models \psi((\val^{m_0})^{G'}(P_0(a)),\dots,(\val^{m_{h-1}})^{G'}(P_{h-1}(a)),\theta),\]
if and only if 
\[G' \models \psi((\val^{m_0})^{G}(P_0(a)),\dots,(\val^{m_{h-1}})^{G}(P_{h-1}(a)),\theta),\]
if and only if
\[G \models \psi((\val^{m_0})^{G}(P_0(a)),\dots,(\val^{m_{h-1}})^{G}(P_{h-1}(a)),\theta).\]

\end{proof}

\begin{corollary}\label{CorollaryGroupEquivalentMaximalGroup}
    Assume that $G$ satisfies (UR) such that $(G,\val)$ satisfies (M). Then $G$ is elementarily equivalent to
    \[\hprod_{\gamma\in \Gamma\setminus \{\infty\}} R_\gamma',\]
    where $R_\gamma'$ is an archimedean group elementarily equivalent to the regular ordered abelian group $R_\gamma$.
    In particular, $G$ is equivalent to a maximal ordered abelian group satisfying (U).
\end{corollary}

\begin{example}
    The ordered abelian group $G:= \hprod_{\omega} \mathbb{Z}$ eliminates quantifiers in the language 
    \[  \{ (G,+,-,0,(\equiv_m,\equiv_{\bullet m} k_\bullet, =_\bullet k_\bullet)_ {k,m\in \mathbb{N}} ),(\omega\cup \{\infty\},0,s,<,\infty), (\val^n)_{n\in \mathbb{N}}\}\]
    where $s$ is the successor function in $\omega$ (extended by $s(\infty)=\infty$). 
\end{example}

\subsection{Stably embedded ordered abelian groups with interpretable regular valuation}\label{SS:MainTheorem}
We prove in this section our main theorem (Theorem~\ref{TheoremCharacterisationStableEmbeddednessCaseInterpretableArchimedeanSpineAbsolut}), characterizing which ordered abelian groups satisfying (M) and (UR) are stably embedded. %Let $G$ be an ordered abelian group satisfying (M) and (UR). 
%We show first that if $G$ is stably embedded then it is maximal. %Then, we characterise when such ordered abelian group is stably embedded.

\begin{proposition}\label{PropositionStablyEmbeddedImpliesMaximal}
Let $G$ be an ordered abelian group satisfying (UR) such that $(G,\val)$ satisfies (M). Let $G\preccurlyeq G'$ be  such that $G \preccurlyeq^{st} G'$. Then $(G,\val)$ is maximal (equivalently, pseudo-complete) in $G'$, and $(G/nG,\val^n)$ is pseudo-complete in $(G'/nG',\val^n)$ for every $n$.

In particular, if $G$ is stably embedded, then $G$ is maximal (equivalently, pseudo-complete) and $(G/nG,\val^n)$ is pseudo-complete for all $n$.
\end{proposition}

\begin{proof}
The equivalence between maximal and pseudo complete for $\val$ follows from Theorem \ref{TheoremMaximalityEquivalentPseudo-completeness}. We (definably) expand $G$ to the many-sorted structure $\mathcal{G}=(G,\Gamma,\val)$. Fix $n\in \mathbb{N}$. We need to show that $(G/nG,\val^n)$ is pseudo-complete in $(G'/nG',\val^n)$. By contradiction, suppose that this is not the case, so there exists a pseudo-Cauchy sequence $(g_i)_{i \in I}$ in $G$ pseudo-converging for $\val^n$ to $a\in G'$ but with no pseudo-limit in $G$ (for $\val^n$)\footnote{The notion of pseudo-limit and pseudo convergence extends naturally to pre-valuations.}. Then, for any $g \in G$, there is $i\in I$ such that $\val^n(a-g)<\val^n(a-g_i)$, since by assumption $g$ is not a pseudo-limit of $(g_i)_{i\in I}$. In particular, we have that $\val^n(a-g)=\val^n(g-g_i) \in \Gamma^n_G$. Therefore, for any $g \in G$, we can consider the ball $B_g:=\Set{h \in G \mid \val^n(a-h) \ge \val^n(a-g)}$. We will show that $\bigcap_{g \in G} B_g \neq \emptyset$. \\ Since $G$ is stably embedded in $G'$, the subset
\[
\Set{(h,g) \in G^2 \mid \val^n(a-h) \geq \val^n(a-g)}
\]
of $G^2$ may be defined with a formula $\phi(x,y,\bar{c})$, where $\phi(x,y,\bar{z})$ is a formula without parameters and $\bar{c}$ is a tuple from $G$. Moreover, we may assume that, for any tuple of parameters $\bar{c}$, the non-empty instances $\phi(x,g,\bar{c})$, with $g \in G$, form a nested family of closed balls (in the sense of the pre-valuation $\val^n$). Notice that the property of $\phi(x,y,\bar{c})$ to define a nested family of closed balls is first-order expressible. By Corollary~\ref{CorollaryGroupEquivalentMaximalGroup}, $G$ is elementarily equivalent to a maximal group $\widetilde{G}$. Then by Theorem~\ref{TheoremMaximalityEquivalentPseudo-completeness} and Proposition~\ref{PropPseudoCompleteModm}:
\[
\widetilde{G} \models \forall \bar{z} \exists x \forall y (\exists w \phi(w,y,\bar{z}) \rightarrow \phi(x,y,\bar{z})),
\] 
and so also
\[
G \models \forall \bar{z} \exists x \forall y (\exists w \phi(w,y,\bar{z}) \rightarrow \phi(x,y,\bar{z})),
\] 
since $G \equiv \widetilde{G}$. In particular, we have that $\bigcap_{g \in G}B_g \neq \emptyset$. Let $b \in \bigcap_{g \in G} B_g$. It follows that $b$ is a pseudo-limit of $(g_i)_{i \in I}$ in $G$. This is a contradiction, proving the first part.

To prove the second part, let $G$ be stably embedded and assume for contradiction that there is a pseudo-Cauchy sequence $(g_i)_{i \in I}$ in $(G/nG,\val^n)$ without pseudo-limit in $G/nG$. By compactness, we find $G\preccurlyeq G'$ such that $G'/nG'$ contains a pseudo-limit of $(g_i)_{i \in I}$. As $G\preccurlyeq^{st} G'$ by assumption, the first part leads to a contradiction.
\end{proof}

We will now prove a slight strengthening of a part of the proposition which we will need later.

\begin{lemma}\label{LemmaPC-CutNondefinable}
Let $G$ be an ordered abelian group satisfying (UR) such that $(G,\val)$ satisfies (M). Let $a$ be a pseudo-limit (in some elementary extension of $G$) of a pseudo-Cauchy sequence $(g_i)_i$ in $G$ without pseudo-limit in $G$, and let $(L,R)$ be the cut of $G$ defined by $a$. Then $L$ is not definable in $G$.
\end{lemma}

\begin{proof}
Let $a$ and $(L,R)$ be as in the statement. Assume for contradiction that $L$ is definable. Note that $\val(g-a)\in\Gamma_G$ for all $g\in G$.
\begin{claim}
    The map $\gamma:g\mapsto\gamma(g):=\val(g-a)$ is definable in $G$.
\end{claim}
\begin{proof}[Proof of the claim]
Let $g\in G$. Choose $g'\in G$ such that $\gamma(g)<\gamma(g')$. Then $B':=B_{\gamma(g')}(g')$ is a proper (closed) subball of $B:=B_{\gamma(g)}(g)$. Note that (closed) balls are convex subsets of $G$. As the non-trivial ordered abelian group $B_{\gamma(g')}(0)/B_{\gamma(g)}(0)$ does not have a maximal or a minimal element, $B'$ is neither an initial nor a final segment of $B$. Thus, by convexity and as $B'$ (interpreted in the elementary extension) contains $a$, $B\cap L\neq \emptyset\neq
B\cap R$.

If $g\in L$, it follows that $\val(g-a)=\val(g-h)$ for all sufficiently small $h\in R$. Similarly, if $g\in R$, then $\val(g-a)=\val(g-h)$ for all sufficiently large $h\in L$. This shows the claim. \end{proof}

By the claim, the nested family of closed balls $\{B_{\gamma(g)}(g) : g\in G\}$ is uniformly definable in $G$. As in the proof of Proposition~\ref{PropositionStablyEmbeddedImpliesMaximal}, it follows from (M) that $G$ contains a pseudo-limit of $(g_i)_{i}$, contradicting the assumption. 
\end{proof}

Now we are able to prove the relative version of our main theorem.

\begin{theorem}\label{TheoremCharacterisationStableEmbeddednessCaseInterpretableArchimedeanSpineForPairs}
Let $G$ be an ordered abelian group satisfying (UR) such that $(G,\val)$ satisfies (M). Let $G\preccurlyeq G'$ be an elementary extension. Then
$G \preccurlyeq^{st} G'$ if and only if the following conditions hold:
\begin{itemize}
    \item The valued abelian group $(G,\val)$ is maximal in  $(G',\val)$, i.e., there is no intermediate immediate extension of $G$ in $G'$ (equivalently $(G,\val)$ is pseudo-complete in $(G',\val)$),
    \item for every $n>1$, $(G/nG,\val^n)$ is pseudo-complete in $(G'/nG',\val^n)$, 
    \item for every $\gamma\in \Gamma$, $R_\gamma \preccurlyeq^{st} R'_\gamma$ (as an ordered abelian group),
    \item $\Gamma \preccurlyeq^{st} \Gamma'$ (as a coloured chain).
\end{itemize}
\end{theorem}

Recall that, by Theorem~\ref{TheoremMaximalityEquivalentPseudo-completeness}, $(G,\val)$ is maximal in $(G',\val)$ if and only if it is pseudo-complete in $(G',\val)$. We do not know if this equivalence holds for the quotients $(G/nG,\val^n)$ as well.

\begin{proof}
(Left-to-right) By Proposition~\ref{PropositionStablyEmbeddedImpliesMaximal}, $(G,\val)$ is maximal in $(G',\val)$ and $(G/nG,\val^n)$ is pseudo-complete in $(G'/nG',\val^n)$ for every $n>1$. To see that $R_\gamma$ is stably embedded in $R_\gamma'$ for any given $\gamma \in \Gamma_G$, it is enough to note that any trace of an $R_\gamma'$-definable subset on $R_\gamma$ is definable with parameters from $G$, since $G\preccurlyeq^{st}G'$, and so definable in the ordered abelian group $R_\gamma$, as $R_\gamma$ is stably embedded and pure in $G$ by Corollary~\ref{Cor:Chain-Rib-StEmb}(2). Stable embeddedness of $\Gamma$ in $\Gamma'$ follows in the exact same way, using Corollary~\ref{Cor:Chain-Rib-StEmb}(1).

\smallskip

%Let $a$ be a tuple in $(G' \setminus G)^{\vert a \vert}$ with $\val(a_i)=\gamma$ for all $i<\vert a \vert$. By hypothesis, $G$ is stably embedded in $G'$. %Then, the type of $a$ over $G$ is definable and so is the type of $a \mod V_\gamma$ over $R_\gamma$. As the induced structure on $R_\gamma$ is given by the %pure structure of ordered group, the defining scheme of $\tp_G(a \mod V_\gamma/ R_\gamma)$ is given by ordered group formulas. Since $a\mod V_\gamma$ was %chosen arbitrarily, $R_\gamma$ is stably embedded in $R_\gamma'$ as ordered abelian groups. %It follows that $R_\gamma$ is stably embedded for any $\gamma \in %\Gamma$.

%Similarly, we show that $\Gamma$ is stably embedded in $\Gamma'$: consider an element $a\in G'$ of value $\gamma$. Since the type of $a$ over $G$ is definable and $\Gamma$ is purely stably embedded, the type of $\gamma=\val(a)$ over $\Gamma$ is definable, and the defining scheme can be given in the language of colour chains $(\Gamma, (C_\phi)_{\phi \in \mathcal{L}_{\text{oag}} })$. Therefore, every $1$-type over $\Gamma$ is definable. Hence, by Corollary~\ref{FactStablyEmbeddedChainIFFAllCutDefinable}, $\Gamma$ is stably embedded in $\Gamma'$. \\

(Right-to-left) We show first that every 1-type over $G$ realised in $G'$ is definable. Then, at the end of the proof, we deduce that all types realised in $G'$ are definable, using an argument similar to the case of regularly ordered abelian groups.   Now, let $a\in G'\setminus G$ be a singleton. 
     Define 
     $$\Theta:=\{\val^n(a-g) \ \vert \ n \in \mathbb{N}, g\in G \} \subseteq \Gamma'.$$ 
    We will show that $\tp(a/\Theta G)$ is definable. Then, since by assumption (and pure stable embeddedness of $\Gamma$ in $G$, see Corollary~\ref{Cor:Chain-Rib-StEmb}) $\tp(\beta/G)$ is definable for all finite tuples $\beta \in \Theta$, we deduce that $\tp(a/G)$ is definable. 

\begin{claim}\label{ClaimBestApproximationCaseOneValuation}
        For every $m\in \mathbb{N}$ and $n\in\Z$ there is $a^m_n \in G$ such that $\beta^m_n := \val^m(na-a^m_n)$ is maximal in the set of values $\Delta^m_n=\{\val^m(na-g) \mid g \in G\}$.
            We say that $a^m_n$ is a (representative of a) \emph{best approximation} of $na \mod m$ in $G/mG$. 
\end{claim}

\begin{proof}[Proof of the claim]       
        By hypothesis, $(G/mG,\val^m)$ is pseudo-complete in $(G'/mG',\val^m)$ for all $m\geq0$, so for all $n\in\Z$ the set $\Delta^m_n\subseteq\Gamma'$ admits a maximum $\beta^m_n=\val^m(na-a^m_n)$, where $a^m_n\in G$.
\end{proof}

For all $m\in\mathbb{N}$, $n\in \mathbb{Z}$ and $g\in G$ we have that
        \begin{align}
        \label{EquationBestApproximationALL}
         \val^m(na-g)=\min\{\val^m(a^m_n-g),\beta^m_n\},
        \end{align}
        as follows from the maximality of $\beta^m_n\in\Delta^m_n$ and the ultrametric triangular inequality.

By Theorem~\ref{TheoremQuantifierEliminationOneValuation}, a formula $\phi(x,\bar{g},\bar{\theta})$ in the language $\mathcal{L}$ with parameters in $G \cup \Theta$ is a finite Boolean combination of formulas of the form:
	\begin{itemize}
		\item[a)] $nx-g > 0$,
		\item[b)] $nx-g\equiv_{\bullet m} k_\bullet$,
		\item[c)] $nx-g =_\bullet k_\bullet$,
		\item[d)] $\psi(\val^{m_0}(n_0x-g_0),\dots,\val^{m_{h-1}}(n_{h-1}x-g_{h-1}),\theta_0,\dots,\theta_{h'-1})$,
	\end{itemize}
where $\psi$ is a formula in $(\Gamma,(C_{\phi})_{\phi\in \mathcal{L}_{\text{OAG}}},<)$,  $n,n_0,\dots,n_{h-1}\in\Z$, $k \in \Z \setminus \{0\}$, $m_0,\dots,m_{h-1} \in \mathbb{N}$, $m \in \mathbb{N} \setminus \{0\}$ and $g,g_0,\dots,g_{h-1} \in G$, $\theta_0,\dots,\theta_{h'-1} \in \Theta$.  We may analyse the definability of each type of formula separately. Notice right away that, by (\ref{EquationBestApproximationALL}), the set 
	\[
	D:= \{(\bar{g},\bar{\theta}) \in G^{h} \times \Theta^{h'} \ \vert \ \psi(\val^{m_0}(n_0a-g_0),\dots, \val^{m_{h-1}}(n_{h-1}a-g_{h-1}),\theta_0,\dots,\theta_{h'-1})\}
	\]
has a definition with parameters in $G \cup \{ \beta_n^m\mid n\in\Z,m\in\mathbb{N}\}$.

\begin{claim}\label{ClaimA}
The set
	\[
	A:= \{ g \in G \ \vert \ na-g>0\}
	\]
is definable with parameters in $G\cup \Theta$. 
\end{claim}
\begin{proof}[Proof of the claim]
We may assume $n\neq0$. Let $g\in G$, and set $\beta:=\val(na-g)$. Then, the sign of $na-g$ is the sign of $na-g \mod V_{\beta}'$, since all elements of $V_{\beta}'$ are smaller than $na-g$ in absolute value.  By (\ref{EquationBestApproximationALL}), $\beta=\min\{\val(a^0_n-g),\beta^0_n\}$ and we may distinguish three cases:
\begin{itemize}
\item $\beta=\val(na^0-g)<\beta^0_n$. Then since $\val( (na-g)-(a^0_n-g) )= \beta^0_n$, we have  $na-g \mod V_{\beta}'=a^0_n-g \mod V_{\beta}'$. Therefore, in that case, $g \in A$ if and only if $a^0_n-g>0$. 
\item $\beta = \beta^0_n < \val(a^0_n-g) $. Then  $na-g \mod V_{\beta}'=(na-g)-(a^0_n-g) \mod V_{\beta}'= na-a^0_n\mod V_{\beta}'$. Thus, the sign of $na-g$ is the sign of $na-a^0_n$. 

\item $\val(a^0_n-g)=\beta^0_n=\beta$. This case is possible only if $\beta^0_n\in \Gamma$. In $R'_\beta$ we have 
$$na-g \mod V_\beta' =na-a^0_n\mod V_\beta' +a^0_n-g\mod V_\beta',$$
so we get 
$$na-g >0\Leftrightarrow na-g\mod V_\beta'>0 \Leftrightarrow g-a^0_n\mod V_\beta'<na-a^0_n\mod V_\beta'.$$
Since $R_\beta\preccurlyeq^{st} R_\beta$ by assumption, the cut of $na-a^0_n$ over $R_\beta$ is definable in $R_\beta$, and so by interpretability we find an $\mathcal{L}$-formula 
$\psi(x)$ with parameters from $G$ such that for all $\tilde{g}\in G $ of value $\beta$ one has $G\models \psi(\tilde{g})$ if and only if $\tilde{g}\mod V_\beta'<na-a^0_n$.
\end{itemize}
Therefore, $A$ is definable with parameters in $G \cup \{\beta^0_n\}$ and is given by 
\[ 
\big ( \val(a^0_n-x)<\beta_n^0 \land a^0_n-x>0 \big) \lor \big(\val(a^0_n-x)>\beta_n^0 \land \epsilon \big) \lor \big(\val(a_n^0-x)=\beta_n^0 \land \psi(x-a^0_n)\big).
\]
where $\epsilon\in \{\bot, \top\}$ is the truth value of $na-a^0_n>0$.
\end{proof}

\begin{claim}
The set
	\[
	B:= \{ g \in G \ \vert \ na-g \equiv_{\bullet m} k_\bullet \}
	\]
is definable with parameters in $G \cup \Theta$.
\end{claim}
\begin{proof}[Proof of the claim]
 Similarly to the proof of Claim~\ref{ClaimA}, we show that one can find $\psi_1(x), \psi_2(x), \psi_3(x)$ with parameters in $G$ 
such that $B$ is given by the following disjunction
\[ 
\big ( \val^m(a^m_n-x)<\beta^m_n \land \psi_1(x) \big) \lor \big(\val^m(a^m_n-x)>\beta^m_n \land \psi_2(x)\big) \lor \big(\val^m(a^m_n-x)=\beta^m_n \land \psi_3(x)\big).
\]
In particular, $B$ is definable with parameters in $G \cup \{\beta^m_n\}$. Let $g \in G$, and set $\beta:=\val^m(na-g)=\min\{\val^m(a^m_n-g),\beta^m_n\}$. We distinguish the following cases:

\begin{itemize}
    \item $\beta=\val^m(a^m_n-g)<\beta^m_n$. Observe that since $\val^m(na-a^m_n)>\beta$,  $na-g \mod V_\beta+mG= a^m_n-g \mod V_{\beta}+mG$. In particular, $na-g \equiv_{\bullet m} k_\bullet$ if and only if $a^m_n-g \equiv_{\bullet m} k_\bullet$, so we may set $\psi_1(x):=a^m_n-g \equiv_{\bullet m} k_\bullet$. 
    
    \item $\beta=\beta^m_n<\val^m(a^m_n-g)$. Then we have $na-g \mod V_{\beta}+mG= na-a^m_n \mod V_{\beta^m_n}+mG$. In particular, $na-g \equiv_{\bullet m} k_\bullet$ if and only if $na-a^m_n \equiv_{\bullet m} k_\bullet$, which depends only on $a$. Then, if $na-a^m_n \equiv_{\bullet m} k_\bullet$, we set $\psi_2(x):= \top$; otherwise, we set $\psi_2(x):=\bot$. 
    \item $\beta=\val^m(a^m_n-g)=\beta^m_n \in \Gamma$ and $R_{\beta}=R_{\beta^m_n}$ is not discrete. Then, trivially, there is no $g \in G$ such that $\val^m(a^m_n-g)=\beta$ and $g \in B$. We set $\psi_3(x):=\bot$. 
    \item $\beta=\val^m(a^m_n-g)=\beta^m_n\in \Gamma$ and $R_{\beta}=R_{\beta^m_n}$ is discrete. By (M), there is $a'\in G'$ such that $a'\equiv_m na-a^m_n$ and $\val^m(na-a_n^m)=\val(a') $. Then, $na-g \equiv_{\bullet m} k_\bullet$ if and only if $a'+a_n^m-g \mod V_{\beta} \equiv_m k_{\beta} \mod V_{\beta}$ where $k_{\beta}$ denotes a representative in $G$ of $k$ times the minimal positive element of $R_{\beta}$. Since $R_{\beta}\preccurlyeq^{st}R'_{\beta}$ and the rib $R_{\beta}$ is interpretable in $G$, there is an $\mathcal{L}$-formula $\psi_3'(x)$ with parameters in $G$  defining the set of elements $\tilde{g}$ of valuation $\beta$ such that $a'+\tilde{g} \mod V_{\beta} \equiv_m k_{\beta} \mod V_{\beta}$. We may then set $\psi_3(x):=  \exists x' x'\equiv_m x \wedge \psi_3'(a_n^m-x')$. Since $G$ satisfies (M), $\psi_3(g)$ holds if and only if $a'+a^m_n-g \equiv_{\bullet m} k_\bullet$ and if and only if $na-g\equiv_{\bullet m} k_\bullet$, as wanted.
    \end{itemize}
\end{proof}
\begin{claim}
The set
\[
C:= \{ g \in G \ \vert na-g =_\bullet k_\bullet \}
\] 
is definable with parameters in $G \cup \Theta$.
\end{claim}
\begin{proof}[Proof of the claim]
 Similarly to the previous claims, we will find $\psi_1(x), \psi_2(x), \psi_3(x)$ with parameters in $G$ 
such that $C$ is given by the following disjunction
\[ 
\big ( \val(a_n^0-x)<\beta_n^0 \land \psi_1(x) \big) \lor \big(\val(a_n^0-x)>\beta_n^0 \land \psi_2(x)\big) \lor \big(\val(a_n^0-x)=\beta_n^0 \land \psi_3(x)\big).
\]
In particular, $C$ is definable with parameters in $G \cup \{\beta_n^0\}$. Fix $g\in G$ and set $\beta:=\val(na-g)$. Again by (\ref{EquationBestApproximationALL}), we have $\beta=\min\{\val(a_n^0-g),\beta_n^0\}$ and we distinguish the following cases:
\begin{itemize}
    \item $\beta=\val(a_n^0-g)<\beta_n^0$. Then we simply observe that $na-g \mod V_{\beta}=a_n^0-g +na-a_n^0 \mod V_{\beta}=a_n^0-g \mod V_{\beta}$. In particular, $na-g =_\bullet k_\bullet$ if and only if $a_n^0-g =_\bullet k_\bullet$, so we may set $\psi_1(x):=a_n^0-x =_\bullet k_\bullet$. 
    \item $\beta=\beta_n^0<\val(a_n^0-g)$. Then we have $na-g \mod V_{\beta}=na-a_n^0) \mod V_{\beta_n^0}$. In particular, $na-g =_\bullet k_\bullet$ if and only if $na-a_n^0 =_\bullet k_\bullet$, which depends only on $a$. Then, if $na-a_n^0 =_\bullet k_\bullet$, we set $\psi_2(x):= \top$; otherwise, we set  $\psi_2(x):=\bot$. 
    \item $\beta=\val(a_n^0-g)=\beta_n^0$ and $R_{\beta}=R_{\beta^m_n}$ is not discrete. Then, trivially, there is no $g' \in G'$ (so in particular no $g\in G$) such that $na-g' \mod V'_\beta$ is $k$ times the minimal element of $R'_\beta$. We set $\psi_3(x):=\bot$. 
    \item $\beta=\val(a_n^0-g)=\beta_n^0$ and $R_{\beta}=R_{\beta^m_n}$ is discrete.  In particular, $\beta \in \Gamma$. Since $R_{\beta}\preccurlyeq^{st}R'_{\beta}$ and the rib $R_\beta$ is interpretable in $G$, we find a $\mathcal{G}$-formula $\psi_3(x)$ with parameters in $G$, defining the set of $\tilde{g}\in G$ of value $\beta$ such that $na-\tilde{g} = k_{\beta} \mod V_{\beta}$, where $k_{\beta}$ denotes a representative in $G$ of $k$ times the minimal positive element of $R_{\beta}$. 
\end{itemize}
\end{proof}
Now let $\bar{a}=(a_0,\dots,a_{k-1})$ be any tuple of new elements in $G'$. Then $\tp(\bar{a}/G)$ is determined by the following set of formulas:
\[
\bigcup_{z_0,\dots,z_{k-1} \in \Z} \tp(\underset{i<k}{\sum}z_ia_i / G)  \cup \tp_{\Th(\Gamma_G)}(\Theta(\bar{a})/ \Gamma_G), 
\]
where $\Theta(\bar{a})=\bigcup_{z_0,\dots,z_{k-1} \in \Z}\Theta(\underset{i<k}{\sum}z_ia_i)$. Since any of these types is definable over $G$, so is $\tp(\bar{a}/G)$. This concludes the proof.
\end{proof}

\begin{theorem}\label{TheoremCharacterisationStableEmbeddednessCaseInterpretableArchimedeanSpineAbsolut}
Assume that $G$ satisfies (UR) and that $(G,\val)$ satisfies (M). Then $G$ is stably embedded if and only if it is maximal, its regular ribs ${(R_\gamma,+,0,<)}$ are stably embedded as ordered abelian groups and its regular spine ${(\Gamma,(C_\phi)_{\phi\in \mathcal{L}_{\text{oag}}},<)}$ is stably embedded as a coloured chain. 
\end{theorem}
\begin{proof}
(Right-to-left) By Proposition~\ref{PropPseudoCompleteModm} and Theorem~\ref{TheoremMaximalityEquivalentPseudo-completeness}, $(G/nG,\val^n)$ is pseudo-complete for every $n$, as $G$ is maximal. Thus this  implication follows from Theorem~\ref{TheoremCharacterisationStableEmbeddednessCaseInterpretableArchimedeanSpineForPairs}.

(Left-to-right) Assume that $G$ is stably embedded. We must show that for all $\gamma$ in $\Gamma$, $R_\gamma$ is stably embedded. Consider a proper elementary extension $\hat{R}$ of $R_\gamma$ in $\mathcal{L}_{\text{oag}}$.  There exists an elementary extension $\mathcal{G}'=(G',\Gamma',\val)$ of $\mathcal{G}=(G,\Gamma,\val)$ such that $R'_\gamma \succeq \hat{R} \succ R_\gamma$. It follows from $G \preccurlyeq^{st} G'$ and Theorem~\ref{TheoremCharacterisationStableEmbeddednessCaseInterpretableArchimedeanSpineForPairs} that $R_\gamma\preccurlyeq^{st}R'_\gamma$, and so in particular $R_\gamma\preccurlyeq^{st}\hat{R}$. A similar argument shows that $\Gamma$ is stably embedded as a coloured chain. Finally, maximality of $G$ follows from Proposition \ref{PropositionStablyEmbeddedImpliesMaximal}, as every pseudo-Cauchy sequence in $G$ without pseudo-limit in $G$ admits a pseudo-limit in some elementary extension $G'\succcurlyeq G$.
\end{proof}

In Subsection~\ref{S:finite-rk} (resp. Subsection~\ref{ColouredChains}) we have characterised stably embedded regular ordered abelian groups (resp. stably embedded coloured chains). Therefore, combining   Proposition~\ref{PropArchimedeangroups} and Corollary~\ref{FactStablyEmbeddedChainIFFAllCutDefinable} with Theorem~\ref{TheoremCharacterisationStableEmbeddednessCaseInterpretableArchimedeanSpineAbsolut}, we get the following corollary:

\begin{corollary}\label{CorMainTheorem1}
Let $G$ be an ordered abelian group satisfying (UR) such that $(G,\val)$ satisfies (M). Then, $G$ is stably embedded if and only if  the following properties hold:
\begin{enumerate}
\item $(G,\val)$ is maximal.
\item For every $\gamma \in \Gamma_G$, either $R_\gamma \cong \Z$, or $R_\gamma$ is densely ordered and $\div(R_\gamma) \cong \mathbb{R}$. 
\item All cuts of $(\Gamma,(C_{\phi})_{\phi\in \mathcal{L}_{\text{oag}}},<)$ are definable.
\end{enumerate}
Moreover, in this case $G$ satisfies (U).
\end{corollary} 
The last statement follows from the fact that every stably embedded regular ordered abelian group is archimedean. Thus, the regular valuation needs to coincide with the archimedean one if $G$ is stably embedded. Since a regular ordered abelian group is stably embedded if and only if all cuts are definable, we can be even be more concise and state the following corollary:

\begin{corollary}\label{CorMainTheorem2}
Let $G$ be an ordered abelian group satisfying (UR) such that $(G,\val)$ satisfies (M). Then, $G$ is stably embedded if and only if all cuts of $G$ are definable.
 \end{corollary}

\begin{proof}
    We already saw that if $G$ is stably embedded, in particular all cuts are definable. Conversely, assume that all cuts of $G$ are definable. We need to check that all the conditions of   Corollary~\ref{CorMainTheorem1} hold in $G$.
    
    First, all cuts of $\Gamma$ are definable in $G$, since their preimages under $\val$ give rise to cuts of $G$, which are definable by assumption. Thus, all cuts of $\Gamma$ are definable in $(\Gamma,<,(C_\phi)_{\phi\in \mathcal{L}_{\text{oag}}})$ by pure stable embeddedness. %From Fact~\ref{FactStablyEmbeddedChainIFFAllCutDefinable} it follows that $(\Gamma,<,(C_\phi)_{\phi\in \mathcal{L}_{\text{oag}}})$ is stably embedded. 
    Similarly, it follows that all cuts of $R_\gamma$ are definable in $R_\gamma$, and since $R_\gamma$ is regular, it must be archimedean and stably embedded by Proposition~\ref{PropArchimedeangroups}, so of the required form.  Maximality of $G$ follows from Lemma~\ref{LemmaPC-CutNondefinable}.
    %Assume not, so there is a strict immediate maximal extension $G'$ of $G$. Notice that, by Corollary \ref{CorollaryAKEForOrderedAbelianGroups}, $G\preceq G'$. Let $a\in G'\setminus G$ be any new element. Then $a$ gives rise to a cut $(L_a,R_a)$ of $G$, where $L_a:= \Set{x\in G \mid x<a}$. Since $a$ is an element in an immediate extension, $\Set{\val(x-a) \mid x\in L_a}=\Set{\val(x-y) \mid x\in L_a, y\in R_a}$ admits no maximum. It follows that any co-final sequence in $L_a$ is pseudo-Cauchy, and any pseudo-limit realises the cut $(L_a,R_a)$. Since $L_a$ is definable and $G\preceq G'$, it means we get also in $G'$ a pseudo-Cauchy sequence with no pseudo-limits, which contradicts Theorem \ref{TheoremMaximalityEquivalentPseudo-completeness}.
\end{proof}

    In particular, since cuts are determined by $1$-types, we deduce immediately a Marker-Steinhorn-like statement for all ordered abelian groups satisfying (M) and (UR):

\begin{corollary}\label{Cor:StEmbAllCutsDefinable}
Let $G$ be an ordered abelian group  satisfying (UR) such that $(G,\val)$ satisfies (M). Then, all types over $G$ are definable if and only if all $1$-types over $G$ are definable.
\end{corollary}

\begin{remark}
    In an earlier version of the paper we claimed a relative version of Corollary~\ref{CorMainTheorem2}. Our proof contained a flaw. We decided to leave the statement as a question.
\end{remark}

\begin{question}
  Let $G$ be an ordered abelian group satisfying (UR) such that $(G,\val)$ satisfies (M). Let $G\preccurlyeq G'$ be an elementary extension. Is $G$ stably embedded in $G'$ if all cuts in $G$ realised in $G'$ are definable?   
\end{question}

\section{Examples and counter-examples}\label{S:Examples and Counter-examples}
In this final section we will discuss some examples and give applications of the main theorem. 

\subsection{Applications of the main theorem}
We deduce from our main theorem (Theorem~\ref{TheoremCharacterisationStableEmbeddednessCaseInterpretableArchimedeanSpineAbsolut}) the following examples:

\begin{example}\label{ExampleHprodZStablyEmbedded}
The ordered abelian group $G_1 := \hprod_{i < \omega} \Z$ is stably embedded, and up to isomorphism it is the unique model of its own theory which is stably embedded.
\end{example}

 It is indeed stably embedded since it satisfies all the requirement of Theorem~\ref{TheoremCharacterisationStableEmbeddednessCaseInterpretableArchimedeanSpineAbsolut}: it satifies (U) and so in particular (UR), it is maximal as it is a Hahn product (Fact~\ref{PropositionMaximalityofHahnproducts}), its regular ribs are copies of $\mathbb{Z}$ and are hence stably embedded (Corollary~\ref{Cor:StEmb}), and finally, its regular spine $(\omega,<)$ is stably embedded  (Examples~\ref{ExampleStablyEmbeddedChains}). One can easily see that the lexicographic sum $G_0:=\sum_{i < \omega} \Z$ is the prime model of $T=\Th(G_1)$. 

 Now assume that $H_1\models T$ is stably embedded, so we may assume that $G_0\preccurlyeq H_1$. Since no proper elementary extension of the archimedean spine $(\omega,<)$ is stably embedded and no proper elementary extension of a rib $(\mathbb{Z},+,0,<)$ is stably embedded, necessarily $H_1$ is an immediate extension of $G_0$. But $G_1$ is the only maximal immediate extension of $G_0$ up to isomorphism, so $H_1\cong G_1$.
 
%We give another example of stably embedded ordered abelian group:

\begin{example}
For $i\in\mathbb{R}$ let $R_i$ be the ordered abelian group
\[
\begin{cases}\mathbb{Z} & \text{ if $i$ is rational} \\
\mathbb{R} &
\text{ if $i$ is irrational}.
\end{cases}
\]
Then the Hahn product $G_2 = \hprod_{i \in \mathbb{R}} R_i$ is a stably embedded ordered abelian group. 
\end{example}

Again, one sees that $G_2$ satisfies (U)  and is maximal by Fact~\ref{PropositionMaximalityofHahnproducts}. Its ribs are copies of $\mathbb{R}$ and $\mathbb{Z}$, and are all stably embedded. Its regular spine is given by the coloured chain $(\Gamma, \Gamma^2)=(\mathcal{T}_2,\mathcal{S}_2)=(\mathbb{R} \cup \{\infty\},\mathbb{Q}\cup\{\infty\})$ and it is stably embedded by  Corollary~\ref{FactStablyEmbeddedChainIFFAllCutDefinable}. %Then, the auxiliary sort $(\Gamma,(C_{\phi})_{\phi\in \mathcal{L}_{\text{oag}}},<)$ is interpreted by the coloured chain $(\mathbb{R},\mathbb{Q},<)$, and it is stably embedded in all elementary extension 
We conclude by Theorem~\ref{TheoremCharacterisationStableEmbeddednessCaseInterpretableArchimedeanSpineAbsolut} that $G_2$ is stably embedded.

\subsection{About the induced $m$-valuation}
We saw in Proposition~\ref{PropPseudoCompleteModm} that if a $\mathbb{Z}$-invariant valued group $(G,\val)$ is pseudo-complete, then so is the induced valued group modulo $m$. One may ask if a similar statement holds for pairs:  Suppose that $G \subseteq H$ is a pure embedding of ordered abelian group (e.g., $G \subseteq H$ is elementary), and that $G$ is pseudo-complete in $H$; is then $G/nG$ pseudo-complete in $H/nH$ for all $n\in \mathbb{N}$ ? We note that this does not always hold for a non-elementary extension $H$:
\begin{example}\label{ExampleGMaximalinHGmod2NonMaximalHmod2}
Let $a$ be a non-standard element in an extension $\mathcal{Z}$ of $\mathbb{Z}$ such that $a\equiv_2 1$. Consider $G:= \sum_\omega \mathbb{Z}$ and the pure extension $H$ of $G$ generated by $G$ and $(a,a, \dots)$ (with support equal to $\omega$) inside $\hprod_{i < \omega} \mathcal{Z}$. Let $G$ and $H$ be both endowed with the archimedean valuation. Then the extension $(H,\val)/(G,\val)$ is pseudo-complete, since for any $h \in H \setminus G$ and $g\in G$ one has  $\val(h-g)=0$, so in particular $h$ cannot be a pseudo-limit of a pseudo-convergent sequence in $G$. However $(G/2G,\val^2)$ is not pseudo-complete in $(H/2H,\val^2)$ since the sequence 
\[((\underbrace{1,\dots,1}_n,0,0,\dots))_n\]
in $G/2G$ pseudo-converges to $a+2H$ in $H/2H$. 
\end{example}

We do not know however any such example where $H$ is an elementary extension of $G$.

%\begin{question}
%Let $G \preccurlyeq H$ be an elementary extension of $\mathbb{Z}$-invariant valued abelian groups. Assume that $G$ (thus also $H$) satisfies (\ref{ConditionMaximality})
%for any coset $a \mod mG$ in $G/mG$, there exists $b_a \in G$ such that $a \equiv_m b_a$ and $\val^m(a)=\val(b_a)$;
%and that $G$ is relatively pseudo-complete in $H$. Then, for every $m \in \mathbb{N}_{>0}$,  $(G/mG,\val^m)\subseteq (H/mH,\val^m)$). 
%Is $(G/mG, \val^m)$ the induced valued group modulo $m$ relatively pseudo-complete in $(H/mH,\val^m)$ ? 
%\end{question}

\subsection{Further examples, counter-examples and a question}\label{S:ExamplesandCounterExamples}
The following example suggests that further investigation is required to characterise all stably embedded ordered abelian groups. In particular, we see the hypothesis (UR) is required in Theorem~\ref{TheoremCharacterisationStableEmbeddednessCaseInterpretableArchimedeanSpineForPairs}.

\begin{example}
     Consider for every prime integer $p$ an archimedean group $\mathcal{Z}_{(p)}$ which is stably embedded, not divisible by $p$ and divisible by every prime different from $p$ (such a group exists by Corollary~\ref{Cor:StEmb}(3)). Set $G_3$ to be the Hahn product 
    \[\hprod_\omega \mathcal{Z}_{(p_n)}\oplus \hprod_{\omega^\star}\mathcal{Z}_{(p_n)},\]
    where $(p_n)_{n<\omega}$ is a strictly increasing sequence of prime numbers.
    Then, one may see that $G_3$ is maximal, with finite spines, that all its archimedean ribs (which are equal to the regular ribs) are stably embedded and the set of spines $\mathcal{A}:= \{ \{p_n,p_n^*\}, n\in \omega \}$ is stably embedded (since it admits no proper elementary extension). However, $G_3$ is not stably embedded by Fact~\ref{FactDefinabilityConvexSubgroups}, since $\hprod_\omega \{0\}\oplus \hprod_{\omega^\star}\mathcal{Z}_{(p_n)}$ is a convex subgroup of $G_3$ which is not definable (e.g., by Fact~\ref{FactDelonFarre}).
    
\end{example}
A reason seems to reside in the fact that the union of the spines $\bigcup\mathcal{A} := (\omega+\omega^\star, \{n,n^\star\})$ is not stably embedded (as a single coloured chain).
The next example is a non-maximal group which admits no immediate elementary extension. This leads us to think that it is stably embedded. 
\begin{example}\label{ExampleGMaximalinHGmod2NonMaximalHmod2Bis}
    Consider the ordered subgroup $G_4$ of $\hprod_\omega\mathbb{Z}$ generated by the direct sum $\sum_\omega \mathbb{Z}$ and the element $a:=(2,2,\cdots)$. It has no  proper immediate elementary extension: consider indeed an elementary extension $H_4 \subseteq \hprod_\omega\mathbb{Z}$. 
    If $b:=(b_i)_i \in H_4\setminus G_4$, consider the element $a+2b $. As $a$ is not divisible by $2$ in $G_4$, neither is $a+2b$ in $H_4$. After computation, we can identify $\Gamma_2$ with $\omega+1\cup\{\infty\}$ and we can write $\val^2(a)=\omega$.
    Since  ${G_4}^{[2]}_{\omega}/ 2G_4 = \{a\mod 2G_4,0\}$ is finite and $H$ is an elementary extension, we have ${G_4}^{[2]}_{\omega}/ 2G_4={H_4}^{[2]}_{\omega}/ 2H_4$. In particular, we have eventually that $2+2b_i=2$, or in other words, that $b_i$ is eventually equal to $0$. It follows that $b\in G_4$ which is a contradiction. Thus there is no such $b$ and $H_4=G_4$.
\end{example}

We conclude with a question, aiming to generalise Theorem~\ref{TheoremCharacterisationStableEmbeddednessCaseInterpretableArchimedeanSpineAbsolut} and to encompass the previous two examples:

\begin{question}\label{QuestionCharacterisationStablyEmbeddedOrderedAbelianGroups}
     Is it true that an ordered abelian group $(G,+,0,<)$ is stably embedded if and only if it is maximal (with respect to the regular valuation) in all its elementary extensions, its regular ribs ${(R_\gamma,+,0,<)}$ are stably embedded as  regular ordered abelian groups and its regular spine $(\bigcup_n\Gamma^n,(C_\phi)_{\phi\in \mathcal{L}_{\text{oag}}},<)$ is stably embedded as a (single) coloured chain?
\end{question}

\bibliographystyle{abbrv}
\bibliography{bibtex}

\end{document}